\documentclass[12pt]{article}
\usepackage[cp1251]{inputenc}
\usepackage{amssymb}
\usepackage{amsfonts}
\usepackage{amsthm}
\usepackage{array}
\usepackage{wrapfig}

\usepackage{graphicx}
\DeclareGraphicsExtensions{.jpg}
\makeatletter
\@addtoreset{equation}{section}
\makeatother

\theoremstyle{plain}
\newtheorem{Th}{Theorem}[section]
\theoremstyle{definition}
\newtheorem{Def}[Th]{Definition}
\newtheorem{corollary}[Th]{Corollary}
\newtheorem{lemma}[Th]{Lemma}
\theoremstyle{remark}
\newtheorem{remark}[Th]{Remark}
\theoremstyle{proof}

\begin{document}
\sloppy
\binoppenalty=10000
\relpenalty=10000

\title{Differential geometry in the theory of Hessian operators\footnote{The paper was supported by the RFBR grant 18-01-00472.}}
\author{N. M. Ivochkina\footnote{St.Petersburg State University of Architecture and Civil Engineering,
Russia, 190005, St.Petersburg, Vtoraya Krasnoarmeiskaya ul., 4. E-mail: ninaiv@NI1570.spb.edu.},
N. V. Filimonenkova\footnote{Peter the Great St.Petersburg Polytechnic University,
Russia, 195251, St.Petersburg, Polytechnicheskaya, 29. E-mail: nf33@yandex.ru.}}
\date{}
\maketitle

\begin{abstract}
The paper introduces a new differential-geometric system which originates from the theory of $m$-Hessian operators. The core of this system is a new notion of invariant differentiation on multidimensional surfaces. This novelty gives rise to the following absolute geometric invariants: invariant derivatives of the surface position vector, an invariant connection on a surface via subsurface, curvature matrices of a hypersurface and its normal sections, $p$-curvatures and $m$-convexity of a hypersurface, etc. Our system also produces a new interpretation of the classic geometric invariants and offers new tools to solve geometric problems. In order to expose an application of renovated geometry we deduce an a priori $C^1$-estimate for solutions to the Dirichlet problem for $m$-Hessian equations.

\vskip 0.1in
Key words: smooth surface, invariant differentiation, curvature matrix, $p$-curvature, $m$-convexity, $m$-Hessian equation, kernel of the boundary sub-barrier.
\vskip 0.1in
MSC: 53A07, 53A55, 35J66.
\end{abstract}

\section{Introduction}

The modern theory of fully nonlinear partial differential equations (FNPDE) counts roughly 35 years and was started by the papers
\cite{Ev82}, \cite{Kr83}, \cite{S83}, \cite{I83} \cite{CNS85}, etc. The main goal of this development has been to find differential operators and functional sets which admit correct setting of the Dirichlet problem. The classic example of such development has been discovered in frames of differential geometry, that is Monge -- Ampere operator and the set of convex functions (see for instance \cite{P75}).

This activity produced new algebraic and geometric structures which are per se interesting irrespective of partial differential equations. The principal absolute geometric invariant appeared in 1985 in the paper \cite{CNS85} by L.Cafarelli, L.Nirenberg and J.Spruck, Theorem 3, p.264:
\begin{Th}
The Dirichlet problem
\begin{equation}\sigma_m(\lambda(u_{xx}))=f>0 \quad in\quad \bar\Omega\subset\mathbb{R}^{n},\quad m>1,\label{cns}\end{equation}
$$u=\varphi\quad on \quad \partial\Omega$$
admits a (unique) admissible solution $u\in C^{\infty}(\bar\Omega)$ provided
\begin{equation} \partial\Omega \hbox{ is connected, and at every point } x\in\partial\Omega,\;\sigma_{m-1}(\kappa_1,\dots,\kappa_{n-1})>0.\label{bc}\end{equation}
In case $\varphi\equiv const$, condition (\ref{bc}) is also necessary for existence of a solution in $C^2(\bar\Omega)$.
\end{Th}
Here $\sigma_m$ is the elementary symmetric function of order $m$, $\{\lambda_i(u_{xx})\}_1^n$ is the set of eigenvalues of the Hessian matrix $u_{xx}$, $\{\kappa_i\}_1^{n-1}$ is the set of principal curvatures of $\partial\Omega\subset\mathbb{R}^{n}$.

Notice that for $m=1$ and $m=n$ equations (\ref{cns}) are equivalent to classic Poisson and Monge -- Ampere equations which are differentiable if $u\in C^k$, $k\geqslant3$. For $1<m<n$ presentation (\ref{cns}) is rather restrictive, it does not admit differentiation of the equation. But this complication is artificial, since $\sigma_m(\lambda(u_{xx}))\equiv T_m(u_{xx})$, where $T_m$ is the $m$-trace of matrices. We qualify operators $T_m[u]=T_m(u_{xx})$ and equations $T_m[u]=f$ as $m$-Hessian.

Denote
\begin{equation}\mathbf{k}_p[\partial\Omega]=\sigma_p(\kappa_1,\kappa_2,\dots,\kappa_{n-1}),\quad p=1,\dots,n-1,\label{pk1}\end{equation}
and name $\mathbf{k}_p[\partial\Omega]$ as $p$-{\it curvature} of $\partial\Omega$. We lay $\mathbf k_0\equiv1$ by definition. As is well known, $\mathbf{k}_1[\partial\Omega]$ and $\mathbf{k}_{n-1}[\partial\Omega]$ are the mean and Gauss curvatures of the hypersurface $\partial\Omega$ respectively. Obviously,  they are $C^{k-2}$-smooth if $\partial\Omega\in C^k$, $k\geqslant2$. To extend the smoothness property to $\mathbf{k}_p[\partial\Omega]$ for $1<p<n-1$ one needs to construct some analog of the Hessian matrix for hypersurfaces.

The absolute geometric invariant (\ref{pk1}) was discovered in \cite{CNS85} and actually motivated geometric research in frames of FNPDE. This geometric research continues to bring out new tools and to set up new problems in geometry as well as in FNPDE.

In this paper we give a systematic description of these new geometric tools, integrate them in the respectively corrected classic surface geometry and demonstrate their application to problem (\ref{cns}). Implementing of this programme requires some pedantry and thus our paper keeps to a text-book style. This paper summarizes and advances results from our previous papers \cite{I90}, \cite{IYP12}, \cite{I12}, \cite{IF13}, \cite{IF14fix}, \cite{FB17}, \cite{IF15smfn}, \cite{IF16prep}.

So Section 2 describes initial notions of surface geometry starting from the very beginning, i.e., from the notions of $C^k$-smooth surface $\Gamma^n\subset\mathbb R^N$, $k\geqslant 2$, and of absolute geometric invariant.

Section 3 introduces the key notion to our new approach -- {\it invariant differentiation} onto a surface. In contrast to classic covariant derivatives, invariant derivatives do not depend on parametrization of $\Gamma^n$ and this invariance removes tensors from further consideration. The first-order invariant partial derivatives form a tangent {\it moving frame} on $\Gamma^n$. We also introduce here the notion of {\it $q$-direction} on $\Gamma^n$ as a generalization of $1$-direction, $1\leqslant q\leqslant n$. We show that a smooth embedding $\Gamma^q\subset\Gamma^n$ may be interpreted as a support of $q$-directions on $\Gamma^n$.

Section 4 introduces {\it curvature matrices} $\mathcal K[\Gamma^n]$ of some hypersurface $\Gamma^n\subset\mathbb R^{n+1}$ and establish their connection with classic notions: with the first and the second quadratic forms, the principal curvatures and the principal directions of $\Gamma^n$. We introduce also {\it normal $q$-sectional curvature matrices} as a generalization of the classic notion of normal curvatures. Such extension became possible due to the notion of {\it normal $q$-section} of $\Gamma^n$ and was a natural sequel to introduction of $q$-directions. We see this piece as one of the principal geometric novelties.

Section 5 starts with a brief survey from algebra of $m$-positive matrices which form one of G{\aa}rding cones. L.G{\aa}rding algebraic theory was outlined in 1959, \cite{G59}, and applied to FNPDE in 1985, \cite{CNS85}. It turned out to be the cornerstone of $m$-Hessian operators (\ref{cns}) study, see \cite{IYP12}, \cite{IF13}, \cite{IF14fix}, \cite{FB16}.

In Subsection 5.2 we give the definition of $p$-{\it curvatures} (\ref{pk1}) via curvature matrix as $\mathbf{k}_p[\Gamma^n]=T_m(\mathcal K[\Gamma^n])$. These $p$-curvatures generate the notion of {\it $m$-convexity} for hypersurfaces, $m=0,1,\dots,n$.
Restricted to some closed hypersurface $\Gamma^n$ the $m$-convexity reads as $\mathbf k_m[\Gamma^n]>0$. This notion provides a new stratification of hypersurfaces, starting from just smooth, $m=0$, to strongly convex, $m=n$.

One more geometric novelty is Sylvester criterion of $p$-convexity for hypersurfaces, produced in Subsection 5.3. Sylvester criterion for $m$-positive matrices was discovered rather recently by N.V.Filimonenkova, \cite{F14prep}, and we extend it to hypersurfaces via traces of $q$-sectional curvature matrices.

In Section 6 we demonstrate how the renovated geometric system works in the theory of $m$-Hessian equations. Namely, we derive an a priori estimate of $\|u\|_{C^1(\bar\Omega)}$ for equation (\ref{cns}) by barrier technique. In order to precisely indicate the origin of requirement (\ref{bc}) we introduce the notion of {\it kernel of local sub-barriers} and prove that its existence at $M_0\in\partial\Omega$ is equivalent to the $(m-1)$-convexity of $\partial\Omega$ at $M_0$. This is the first step to construct an a priori estimate of solutions in $C^2$ and to prove the classic solvability of the Dirichlet problem (\ref{cns}) via the continuity method.

The system of geometric tools described in this article is used not only in analysing the problem (\ref{cns}). The concept of $p$-curvature when appeared triggered the following new problems in the intersection of FNPDE and differential geometry. Below we outline some of them to put the system we introduce in perspective.

The first is the Dirichlet problem for $p$-curvature equation. It was studied by FNPDE methods in the papers \cite{CNS}, \cite{T90}. Right then invariant differentiation unnamed appeared as a tool to investigate $p$-curvature equations in the papers \cite{I89}, \cite{I90}. Beyond that, some evolutions of closed convex hypersurfaces were investigated by methods of tensorial differential geometry, see, for instance, \cite{Hs}, \cite{Gr}, \cite{U90}, \cite{An}. It turned out that geometric evolutions may be started as well with $m$-convex initial hypersurfaces, \cite{INT00}, \cite{I04}, \cite{I07}.

Notice that the above geometric activity leaves unattended generalized Minkowski problem. The latter was set up in the book of A.V.Pogorelov, \cite{P75}, Section 4, and in our terminology reads as:

to find restrictions on a given function $\varphi_m=\varphi_m(\xi)$, $|\xi|=1$, such that there exists a closed $C^k$-smooth, $k>2$, convex  hypersurface $\Gamma^n$ satisfying equation
\begin{equation}\mathbf k_{n;n-m}[\Gamma^n]=\frac{\mathbf k_n}{\mathbf k_{n-m}}[\Gamma^n](M)=\frac{1}{\varphi_m(\mathbf n[\Gamma^n])(M)},\quad M\in\Gamma^n,\quad1\leqslant m\leqslant n,\label{Mp}\end{equation}
where $\mathbf n[\Gamma^n](M)$ is the unit normal to $\Gamma^n$.

If $m=n=2$, (\ref{Mp}) coincides with classic Minkowski problem: to reconstruct a closed hypersurface by given Gauss curvature. For $n>2$ A.V.Pogorelov names (\ref{Mp}) as {\it multidimensional} Minkowski problem and presents necessary and sufficient conditions of its solvability in Theorem 1, Section 3, \cite{P75}.  Theorem 2 from Section 4, \cite{P75}, serves the case $m<n$ and it contains two assumptions: the first is necessary for solvability of problem (\ref{Mp}), while the second is sufficient but by far not necessary.

Up to 1975 in differential geometry only convex closed hypersurfaces had been considered. Then it was common to formulate problems in terms of curvature radius and support functions as in \cite{P75}. This is the reason for functions $\varphi_m$ in (\ref{Mp}) to be defined over the unit sphere.

Now it looks natural to reset Minkowski problem for $m$-convex hypersurfaces and to consider, for instance, the following version:

to find restrictions on a given function $f_m=f_m(M)$, $M\in\mathbb R^{n+1}$, such that there exists a closed $C^k$-smooth, $k>2$, $m$-convex hypersurface $\Gamma^n$ satisfying equation
\begin{equation}\mathbf k_m[\Gamma^n](M)=f_m(M)\geqslant\nu>0,\quad M\in\Gamma^n,\quad1\leqslant m\leqslant n.\label{Mpm}\end{equation}

Notice that if $m=n$, $f_n(M)=1/\varphi_n(\mathbf n[\Gamma^n])(M)$, problem (\ref{Mpm}) is equivalent to multidimensional Minkowski problem (\ref{Mp}) solved in \cite{P75}. For $1\leqslant m<n$ the problem (\ref{Mpm}) does not coincide with (\ref{Mp}), which makes unnatural the requirement of convexity of solutions.

We believe that geometric system we present in this paper together with its future extension are necessary to face this challenge.
\vskip 0.1in

The following standard notations are used throughout the paper:

\noindent
$B_r(M_0)$ -- $r$-radius ball centered at $M_0$;

\noindent
${\mathbf 0}=(0,0,\ldots,0)\in\mathbb R^n$; $I$ -- unit matrix;

\noindent
$(x,y)$ -- scalar product of vectors $x,y\in \mathbb R^n$; $x^2=(x,x)$; $|x|=\sqrt{(x,x)}$;

\noindent
$x\times y$ -- outer product of vectors $x,y\in \mathbb R^n$;

\noindent
$C=(c^i_{j})$, $i=1,\dots, p$, $j=1,\dots, n$, -- $(p\times n)$-matrix with elements $c^i_{j}$;

\noindent
${\rm Sym}(n)$ -- linear space of symmetric $(n\times n)$-matrices $S=(s_{ij})$;

\noindent
${\rm O}(n)$ -- group of orthogonal $(n\times n)$-matrices;

\noindent
$\theta_\xi=\left({\partial\theta^i}/{\partial\xi^j}\right)$ -- gradient of vector-valued function $\theta(\xi)\in \mathbb R^n$, $\xi\in\mathbb R^p$, i.e., the Jacobian matrix;

\noindent
${\rm Span}\{l_i\}_1^q$ -- linear span of vectors $l_i\in \mathbb R^n$, $i=1,\dots,q$.

Latin letters $i$, $j$, $k$, $l$, $p$, $m$, $n$, $N$ notate positive integers and summation over repeating upper and lower indices is always assumed.

\section{On surfaces in Euclidean space}

Consider a set $\Gamma^n$ in $N$-dimensional Euclidean space $\mathbb E^N$.

\begin{Def}
Let $1\leqslant n<N$. Some connected set $\Gamma^n\subset \mathbb E^N$ is called a {\it surface of codimension $N-n$} if for every point $M_0\in\Gamma^n$ there are $r>0$, a domain $\Theta\subset \mathbb E^n$ and a homeomorphism $\Theta\rightarrow\Gamma^n\cap B_r(M_0)$.
We call the mapping $\Theta\rightarrow\Gamma^n\cap B_r(M_0)$ a {\it local parametrization} of the surface $\Gamma^n$ in a neighborhood of $M_0$.
\end{Def}
Notice that under Definition 2.1 the sets $\{\Gamma^1\}$, $\{\Gamma^{N-1}\}$ are {\it curves} and {\it hypersurfaces} respectively.

Now we install in $\mathbb E^N$ and $\mathbb E^n$ some cartesian coordinate systems, turn to arithmetic Euclidian spaces $\mathbb R^N$ and $\mathbb R^n$ and fix up  a point $M\in\Gamma^n\cap B_r(M_0)$ by its position vector $X[\Gamma^n](M)=X(\theta)\in\mathbb R^N$:
\begin{equation}X(\theta)=\left(\matrix{x^1(\theta)\cr x^2(\theta)\cr\ldots \cr x^N(\theta)\cr}\right), \quad \theta=\left(\matrix{\theta^1\cr\theta^2\cr\ldots\cr\theta^n}\right)\in\Theta\subset\mathbb R^n.\label{radius}\end{equation}
Assume that $X(\theta)$ is $C^1$-smooth. Introduce its gradient as the Jacobian $(N\times n)$-matrix:
$$X_\theta=\left(\frac{\partial x^i}{\partial\theta^j}\right),\quad i=1,\dots,N,\quad j=1,\dots,n.$$
Denote the columns of the matrix $X_\theta$ by $X_j$ :
\begin{equation}X_j=\frac{\partial X}{\partial\theta^j}=\left(\frac{\partial x^1}{\partial\theta^j}, \frac{\partial x^2}{\partial\theta^j},\ldots , \frac{\partial x^N}{\partial\theta^j}\right)^T,\quad j=1,\ldots,n.\label{grad}\end{equation}
It is well known that the vectors $X_j(M)$ are tangent to the surface $\Gamma^n$ at the point $M$. By Definition 2.1 the mapping (\ref{radius}) is a local parametrization of the surface $\Gamma^n$ if
\begin{equation}\det X_\theta^T X_\theta\neq0.\label{nonded}\end{equation}
It means that at each point of the surface vectors $\{X_j\}_ 1^n$ are linearly independent and may be interpreted as a basis in the tangent $n$-plane, although non-cartesian in general.

\begin{Def}
A surface $\Gamma^{n}\subset\mathbb R^N$ is $C^k$-{\it smooth} if it admits a $C^k$-smooth parametrization (\ref{radius}), (\ref{nonded}) in a neighborhood of each point.
\end{Def}
In this paper only $C^k$-smooth surfaces are under consideration, $k\geqslant 2$.

The existence of at least one $C^k$-smooth parametrization (\ref{radius}) carries out the infinity of such. Indeed, any $C^k$-smooth non-identical mapping
$$\xi=(\xi^1,\xi^2,\dots,\xi^n)^T\in\Xi\subset\mathbb R^n\rightarrow\theta=(\theta^1,\theta^2,\dots,\theta^n)^T\in\Theta\subset\mathbb R^n$$
with $\det\theta_{\xi}\neq0$ presents a new parametrization. We divide  the set of all parametrizations into two classes of equivalence by the sign of $\det\theta_{\xi}$.

\begin{remark} In this paper all characteristics of $\Gamma^n$ are defined mostly in some parametrized neighborhood of a given point $M_0$, which means, strictly speaking, they are characteristics of $\Gamma^n\cap B_r(M_0)$. However, we assign them to the entire surface $\Gamma^n$, assuming the non-degenerate preserving admissibility correspondence between local parametrizations in intersecting neighborhoods.
\end{remark}

Classic differential geometry implies tensorial approach. Here we keep to absolute geometric invariants of surfaces in the following sense.

\begin{Def}
A characteristic of the surface $\Gamma^n$ is an {\it absolute geometric  invariant} if it does not depend on the choice of local parametrization, i.e., it is the same for all equivalent parametrizations.
\end{Def}

For instance, the dimension $n$ of $\Gamma^n$ is an absolute geometric invariant. The property of linear independence of the tangent vectors $X_j$ is an absolute geometric invariant, while these vectors are not. Also, the set of tangent $n$-planes $\{\mathcal L^n(M), M\in\Gamma^n\}$:
\begin{equation}\mathcal L^n[\Gamma^n](M)={\rm Span}\{X_{j}[\Gamma^n](M)\}_1^n\label{tP}\end{equation}
is an absolute geometric invariant.

The main goal of this paper is to introduce some collection of new absolute geometric invariants and to describe the connection of those with classic analogs.

\section{Invariant differentiation}

Describe for the start some properties of the metric tensor. The notion of ``metric tensor'' appeared in frames of Riemannian geometry. Since smooth Riemannian manifolds may be interpreted as surfaces in Euclidean space of sufficiently large dimension, it is natural to extend this notion to surfaces.

\begin{Def}
Let $X(\theta)$ be a parametrization of the surface $\Gamma^n$ in the neighborhood of a given point, $X_\theta$ be the gradient (\ref{grad}). The {\it metric tensor} of this surface is the symmetric $(n\times n)$-matrix
\begin{equation}g[\Gamma^n](\theta)=X_\theta^T X_\theta,\quad g_{ij}=(X_i,X_j),\quad i,j=1,\dots,n.\label{metr}\end{equation}
\end{Def}

Since $g$ is the Gram matrix of linearly independent vectors $\{X_j\}_{1}^n$, the determinant of $g$ is positive and equals to the square of the volume of $n$-dimensional parallelepiped with these vectors as ribs.

In frames of  FNPDE there appeared the algebraic notion of $p$-traces.
\begin{Def}
Let $S\in{\rm Sym}(n)$, $1\leqslant p\leqslant n$. The sum of all $p$-s order principal minors of $\det S$ is called $p$-{\it trace} of $S$ and denoted by $T_p(S)$. By definition $T_0(S)\equiv 1$.
\end{Def}
In particular, $T_1(S)={\rm tr}S$, $T_n(S)=\det S$.
The $p$-traces of the metric tensor have a rather simple geometric interpretation.
Namely, let $1\leqslant i_1<\ldots<i_p\leqslant n$ be an arbitrary set of integers. Denote by $V_{i_1\dots i_p}$ the volume of the parallelepiped with ribs $X_{i_1},\dots, X_{i_p}$. Then
$$T_p(g)=\sum_{i_1<\dots<i_p}V^2_{i_1\dots i_p},\quad  p=1,\dots, n.$$
Notice that the metric tensor and its $p$-traces are not invariant in the sense of Definition 2.4.

\begin{Def} Let $\Gamma^n$ be a $C^k$-smooth surface parametrized by $\theta\in\Theta\subset\mathbb R^n$ and $\Gamma^q\subset\Gamma^n$, $1\leqslant q\leqslant n$, be a $C^k$-smooth surface parametrized by $\xi\in\Xi\subset\mathbb R^q$.
We say that $\Gamma^q$ is a {\it smooth $q$-embedding} into $\Gamma^n$ if there exists a $C^k$-smooth mapping $\theta=\theta(\xi)$ such that
$\det\theta^T_\xi\theta_\xi\neq0$ and $X[\Gamma^q](\xi)=X[\Gamma^n](\theta(\xi))$, $\xi\in\Xi$.
\end{Def}

The metric tensor of some embedded surface $\Gamma^q\subset\Gamma^n$ is determined by the metric tensor of the surface $\Gamma^n$:
\begin{equation}X_{\xi}=X_{\theta}\theta_{\xi},\quad g[\Gamma^q](\xi)=\theta_{\xi}^Tg[\Gamma^n](\theta)\theta_{\xi}.\label{zam}\end{equation}
In the case $q=n$ formula (\ref{zam}) just registers the tensorial nature of $X_{\theta}[\Gamma^n]$, $g[\Gamma^n]$ and means that they are ones and twice covariant tensors respectively.

A key role in our differential-geometric development belongs to the $(n\times n)$-matrices $\tau$, defined by two requirements:
\begin{enumerate}
\item[1)] $\tau=\tau[\Gamma^n]$ is connected with the metric tensor by equality
\begin{equation}g^{-1}[\Gamma^n]=\tau[\Gamma^n]\tau^T[\Gamma^n]\label{tau1};\end{equation}
\item[2)]  if $X(\theta)$, $X(\xi)$ are local parametrizations of $\Gamma^n$, then
\begin{equation}\tau[\Gamma^n](\xi)=\theta_{\xi}^{-1}\tau[\Gamma^n](\theta(\xi)).\label{tau2}\end{equation}
\end{enumerate}

It is obvious that $\tau_0[\Gamma^n]=\sqrt{g^{-1}[\Gamma^n]}\in{\rm Sym}(n)$ satisfies the requirements (\ref{tau1}), (\ref{tau2}) and an arbitrary matrix from the set
\begin{equation}\tau[\Gamma^n](M)=\tau_0[\Gamma^n](M)B,\quad B\in {\rm O}(n),\label{tau3}\end{equation}
also does. Matrices $B$ in (\ref{tau3}) do not depend on parametrization and provide degree of freedom when choosing $\tau =(\tau^i_j)$. Generally speaking, matrices $B$ can depend on points $M\in\Gamma^n$ but in this paper they are constant. Further on we assume that the choice of $B$ has been made and our matrix $\tau[\Gamma^n]$ has been fixed by (\ref{tau3}).

Notice that $(g\tau_i,\tau_j)[\Gamma^n]=\delta_{ij}$, $\tau_j=(\tau^1_j,\tau^2_j,\ldots\tau^n_j)^T$, $j=1,\dots,n$.

\begin{figure}
	\center{\includegraphics[width=8cm]{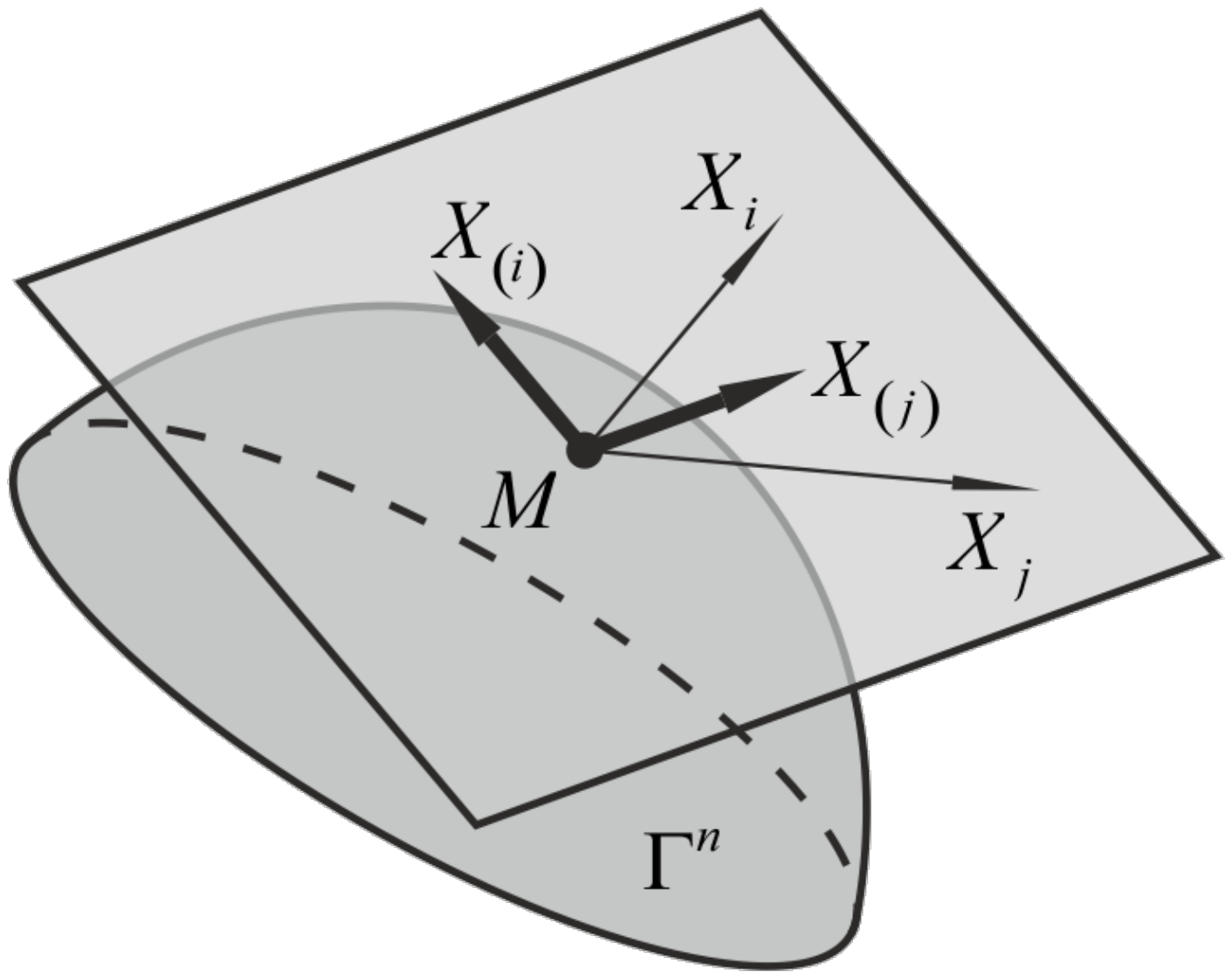} \\ Fig. 1}
\end{figure}

Now we introduce the notion of invariant derivatives by the following definition.
\begin{Def}
We say that the vectors given by the formula
\begin{equation}X_{(j)}=X_k\tau^k_j,\quad j=1,\dots,n,\label{j}\end{equation}
are the first-order {\it invariant partial derivatives} and
the set $\{X_{(j)}(M)\}_{j}^n$, $M\in\Gamma^n$, is a {\it moving frame} on the surface $\Gamma^n$.
We also say that $(N\times n)$-matrix $X_{(\theta)}=X_\theta\tau(\theta)$ is the geometric invariant gradient of the position vector $X[\Gamma^n](M)$.
\end{Def}

It follows from (\ref{j}) that every vector $X_{(j)}$ is a linear combinations of the tangent vectors $\{X_k\}_{1}^n$ and belongs to the tangent $n$-plane (\ref{tP}) (see Fig.1).

Notice that agreement (\ref{tau2}) provides the absolute geometric invariance of $X_{(\theta)}$.

Since $(X_{(i)},X_{(j)})=\delta_{ij}$, the moving frame $\{X_{(j)}(M)\}_{1}^n$ is a cartesian basis along $\Gamma^n$. It smoothly depends on $M\in\Gamma^n$, and one is free to rotate this basis in the tangent planes (\ref{tP}) via sufficiently smooth matrices $B$ in (\ref{tau3}).

Now it is possible to produce the following geometric invariant description of the tangent planes (\ref{tP}):
\begin{equation}\mathcal L^n[\Gamma^n](M)={\rm Span}\{X_{(j)}[\Gamma^n](M)\}_1^n,\quad M\in\Gamma^n.\label{IL}\end{equation}
Denote by $\eta$ some $(n\times q)$-matrix, $1\leqslant q\leqslant n$, such that $\eta^T\eta=I\in{\rm Sym}(q)$.
Introduce a $q$-dimensional sub-plane of (\ref{IL}) by relation:
\begin{equation}\mathcal L^q_\eta[\Gamma^n](M)={\rm Span}\{l_i[\Gamma^n](M)\}_1^q,\quad l_i[\Gamma^n]=X_{(k)}[\Gamma^n]\eta^k_i,\quad i=1,\dots,q.\label{qL}\end{equation}

\begin{Def}
Let $1\leqslant q\leqslant n$. We say that the plane $\mathcal L^q_\eta[\Gamma^n](M)$ from (\ref{qL}) is a {\it $q$-direction} on the surface $\Gamma^n$ at $M$ and the matrix $\eta$ is a {\it cartesian allocator} of this $q$-direction.
\end{Def}

\begin{remark}
For each $q$-direction its cartesian allocator $\eta$ is not uniquely determined. The whole set $\{\eta_0B, B\in {\rm O}(q)\}$ consists of cartesian allocators if $\eta_0$ is one of them.
\end{remark}

Consider two classic examples:
\begin{enumerate}
\item[1)] $1$-direction on $\Gamma^n$ at $M$ is a tangent straight line which consists of all vectors collinear to $l=X_{(k)}\eta^k$, where $\eta\in\mathbb R^n$, $|\eta|=1$, is its cartesian allocator;

\item[2)] $n$-direction on $\Gamma^n$ at $M$ is the whole tangent plane (\ref{IL}) and an arbitrary orthogonal $(n\times n)$-matrix $\eta$ may be taken as its cartesian allocator.
\end{enumerate}

Show now that a smooth $q$-embedding $\Gamma^q\subset\Gamma^n$, see Definition 3.3, may be interpreted as a support of $q$-directions on $\Gamma^n$. Indeed, the moving frame $\{X_{(j)}[\Gamma^q](\xi)\}_{1}^q$ onto $\Gamma^q$ connected with the moving frame $\{X_{(j)}[\Gamma^n](\theta(\xi))\}_{1}^n$ onto $\Gamma^n$ by the equality
\begin{equation}X_{(\xi)}[\Gamma^q](\xi)=\left(X_{(\theta)}[\Gamma^n]\zeta\right)(\xi),\quad\zeta(\xi)=\left(\tau^{-1}[\Gamma^n]\theta_\xi \tau[\Gamma^q]\right)(\xi),\quad\xi\in\Xi. \label{dz}\end{equation}
The $(n\times q)$-matrix $\zeta$ is an absolute geometric invariant, which is a good reason to specify it.
\begin{Def}
We say that the $(n\times q)$-matrix $\zeta$ from (\ref{dz}) is the {\it $q$-connection} on the surface $\Gamma^n$ via $\Gamma^q$.
\end{Def}
Our $q$-connection generates a smooth field of $q$-directions over $\Gamma^n$ in the following sense.

\begin{lemma}
Let $M\in\Gamma^q\subset\Gamma^n$. Then $\zeta(M)$ is a smooth cartesian allocator of $q$-direction given by the tangent plane of $\Gamma^q$ at $M$.
\end{lemma}
\begin{proof}
Relation (\ref{dz}) brings out the identity $\zeta^T\zeta=I\in{\rm Sym}(q)$ for all $M\in\Gamma^q$.
The tangent plane $\mathcal L^q[\Gamma^q](M)$ may be considered as a sub-plane of $\mathcal L^n[\Gamma^n](M)$. Due to (\ref{dz})
$$\mathcal L^q[\Gamma^q](M)={\rm Span}\{X_{(i)}[\Gamma^q](M)\}_1^q,\quad X_{(i)}[\Gamma^q]=X_{(k)}[\Gamma^n]\zeta^k_i.$$
Therefore, $\mathcal L^q[\Gamma^q](M)=\mathcal L^q_\eta[\Gamma^n](M)$ from Definition 3.5 with $\eta=\zeta(M)$.
\end{proof}

Formula (\ref{j}) admits a natural extension to definition of higher-order invariant derivatives. In further proceeding we make use of notation $X_{\theta\theta}$ for $(N\times n\times n)$-array of the second derivatives of the position vector $X[\Gamma^n]=X(\theta)$. The second-order invariant partial derivatives are to be read as
\begin{equation}X_{(i)(j)}=(X_{(i)})_{(j)}=X_{kl}\tau^k_i\tau^l_j+X_k(\tau^k_i)_l\tau^l_j,\quad i,j=1,\dots,n.\label{ij}\end{equation}
They are noncommutative in general, but commutative up to the tangential term.
It follows from (\ref{ij}) that the derivatives
\begin{equation} X_{(ij)}=X_{kl}\tau^k_i\tau^l_j,\quad X_{(\theta\theta)}=\left(X_{(ij)}\right)_{i,j=1}^n=\tau^TX_{\theta\theta}\tau.\label{cij}\end{equation}
are commutative but not invariant.
Notice that commutators of derivatives (\ref{ij}) are absolute geometric invariants, although geometric meaning of those awaits for attention of specialists.

\section{Curvature matrix}

\subsection{Definition of curvature matrix}

In this Section we consider surfaces of codimension 1, i.e., appoint $N=n+1$. These are known as hypersurfaces.

At each point $M$ of some $C^k$-smooth hypersurface $\Gamma^{n}\subset\mathbb R^{n+1}$ there exist two unit normals $\mathbf n^+[\Gamma^n](M)$, $\mathbf n^-[\Gamma^n](M)$, $\mathbf n^+=-\mathbf n^-$. They are geometric invariants. If $X[\Gamma^n](\theta)$, $\theta\in\Theta$, is a local parametrization, the unit normals may be defined via vector-product (exterior product) of tangential vectors:
\begin{equation}\mathbf n[\Gamma^n]=\frac{[X_1,X_2,\dots,X_{n}]}{\left|[X_1,X_2,\dots,X_{n}]\right|}=
[X_{(1)},X_{(2)},\dots,X_{(n)}].\label{n}\end{equation}
It should be reminded that our local parametrizations are divided into two classes of equivalence (see Section 2) and vector (\ref{n}) changes direction to the opposite when passing from one class to the other. That is why the normal vector (\ref{n}) may be considered as an indicator of orientation and makes possible to differ the sides of $\Gamma^{n}$ in ambient space at least locally.

The further differential-geometric notions are related to the $C^k$-smooth, $k\geqslant 2$, two-sided hypersurface $\Gamma^{n}$ oriented by a preferable normal $\mathbf{n}^+=\mathbf{n}^+[\Gamma^n]$. For instance if $\Gamma^{n}$ bounds a domain, $\mathbf{n}^+[\Gamma^n]$ is the interior normal, i.e., it looks into the domain.

For $\Gamma^{n}$ we have the second-order invariant partial derivatives (\ref{ij}) introduced via matrix (\ref{tau3}).

\begin{Def}
We say that the symmetric $(n\times n)$-matrix
\begin{equation}\mathcal K[\Gamma^n]=(X_{(\theta)(\theta)},\mathbf n^+)[\Gamma^n],\quad\mathcal K_{ij}=(X_{(i)(j)},\mathbf n^+),\quad i,j=1,\dots, n,\label{cm}\end{equation}
is the {\it curvature matrix} of $\Gamma^{n}$.
\end{Def}

The symmetry of the curvature matrix is crucial in our applications.

Due to presentation (\ref{ij}), (\ref{cij}) we deduce working formulas for $\mathcal K$:
\begin{equation}\mathcal K[\Gamma^n]=(X_{(\theta\theta)},\mathbf n^+)[\Gamma^n]=\tau^T(X_{\theta\theta},\mathbf n^+)\tau,\label{cms}\end{equation}
$$\mathcal K_{ij}=(X_{(ij)},\mathbf n^+)=(X_{kl},\mathbf n^+)\tau_i^k\tau_j^l,\quad i,j=1,\dots, n.$$

\begin{remark}
For $C^k$-smooth hypersurfaces the metric tensor is $C^{k-1}$-smooth in $\theta$ by definition. Therefore the matrix $\tau_0[\Gamma^n]=\sqrt{g^{-1}[\Gamma^n]}$ is also $C^{k-1}$-smooth, while the curvature matrix with $\tau=\tau_0$ is $C^{k-2}$-smooth in $M\in\Gamma^n$. But in (\ref{cms}) we can take an arbitrary matrix $\tau$ from the family (\ref{tau3}) and Definition 4.1 sets actually the family of curvature matrices
$$\mathcal K=B^T\mathcal K_0[\Gamma^n]B,\quad B\in{\rm O}(n),$$
where $\mathcal K_0$ corresponds to $\tau_0$. The smoothness of these matrices depends, in addition, on the smoothness of $B=B(M)$. For instance, one may choose $B=B(M)$ so that the curvature matrix is diagonal for all $M\in\Gamma^n\cap B_r(M_0)$. However, in this case $\mathcal K[\Gamma^n](M)$ is not more than Lipschitz continuous in general.
\end{remark}

It follows from (\ref{cm}) that curvature matrices are new geometric invariants of $\Gamma^n\subset\mathbb R^{n+1}$ in the sense of Definition 2.4. In addition, they are invariant under orthogonal transformations of the ambient space in view of identity $\tau^T(X_{\theta\theta},\mathbf n^+)\tau=\tau^T(BX_{\theta\theta},B\mathbf n^+)\tau$, $B\in{\rm O}(n+1)$.

Let us try Definition 4.1 for the case $n=1$. Consider a plane curve $\Gamma^{1}\subset\mathbb R^{2}$ and its local parametrization
$$X(\theta)=\left(\matrix{x^1(\theta)\cr x^2(\theta)\cr}\right), \quad \theta\in\Theta\subset\mathbb R^1.$$
In this degenerated case the ``metric tensor'' reads as
$$X_{\theta}=\left(\matrix{x^1_{\theta}\cr x^2_{\theta}\cr}\right)\quad\Rightarrow\quad g=X^T_{\theta}X_{\theta}=|X_{\theta}|^2=(x^1_{\theta})^2+(x^2_{\theta})^2\quad\Rightarrow\quad
\tau_0=\sqrt{g^{-1}}=\frac{1}{|X_{\theta}|}.$$
Obviously, the vector
$$\mathbf n^+=\frac{1}{|X_{\theta}|}\left(\matrix{x^1_{\theta}\cr -x^2_{\theta}\cr}\right)$$
is a unit normal for our curve. Keeping in mind that Definition 4.1 is applicable to the oriented curves, we assume the above unit normal provides a desirable orientation. Then formula (\ref{cms}) leads to
\begin{equation}\mathcal K[\Gamma^1]=\quad\frac{(X_{\theta\theta},\mathbf n^+)}{|X_{\theta}|^2}=\frac{x^1_{\theta\theta}x^2_{\theta}-x^2_{\theta\theta}x^1_{\theta}}{|X_{\theta}|^3}.\label{kriv}\end{equation}
On the other hand, this line is equivalent to
\begin{equation}\mathcal K[\Gamma^1]=\frac{(d^2X,\mathbf n^+)}{(dX,dX)}=\kappa[\Gamma^1],\quad dX=X^\prime d\theta,\quad d^2X=X^{\prime\prime}(d\theta)^2.\label{krd}\end{equation}
Thus, the curvature matrix of a plane curve is a number which is natural to call the {\it curvature} $\kappa[\Gamma^1]$. It almost coincides with the classic definition of curvature of a plane curve but the sign. Namely, in order to avoid the problem of orientation, the classic approach sets the absolute value of $\kappa[\Gamma^1]$ as the curvature. It looks somewhat unnatural. One would better give a thought to the orientation of the curve of his and use definition (\ref{kriv}).

\subsection{Curvature matrix for a graph}

Let a hypersurface $\Gamma^{n}\subset\mathbb{R}^{n+1}$ be the graph of some function
$$x^{n+1}=\omega(x),\quad x=(x^1,x^2,\ldots,x^n)^T\in\Omega\subset\mathbb R^n.$$
Then the position vector reads as
\begin{equation}X[\Gamma^{n}]=\left(\matrix{x\cr \omega(x)\cr}\right).\label{yav}\end{equation}
Introduce the following notations:
$$\omega_{i}=\frac{\partial\omega}{\partial x^i},\quad\omega_{kl}=\frac{\partial^2\omega}{\partial x^k\partial x^l},$$
$$\omega_x=\nabla\omega=(\omega_1,\ldots,\omega_n),\quad \omega_{xx}=\left(\omega_{kl}\right)_{1}^{n}.$$
\begin{lemma}
Assume that $\Gamma^{n}$ is given by (\ref{yav}), where $\omega$ is $C^k$-smooth, $k\geqslant 2$. Then the matrices (\ref{tau3}) read as
\begin{equation}\tau=\tau_0[\Gamma^n]B,\quad \tau_0 =I-\frac{\omega_x\times\omega_x}{\sqrt{1+\omega_ x^2}(1+\sqrt{1+\omega_ x^2})},\quad B\in{\rm O}(n),
\label{tauu}\end{equation}
and curvature matrix (\ref{cm}) has to be chosen from the following two:
\begin{equation}\mathcal{K}^{\pm}[\Gamma^n]=\frac{\pm1}{\sqrt{1+\omega_{x}^2}}\tau^T\omega_{xx}\tau,
\label{krr}\end{equation}
where the sign ``+'' corresponds to the case $(\mathbf{n}^+[\Gamma^n],\mathbf e_{n+1})>0$.
\end{lemma}
\begin{proof}
The lines (\ref{grad}), (\ref{metr}) with vector (\ref{yav}) look as
$$
X_j=\left(0,\ldots,1,\ldots,0,\omega_j\right)^T,\quad j=1,\dots,n,\quad g=I+\omega_x\times\omega_x,\quad g^{-1}=I-\frac{\omega_x\times\omega_x}{1+\omega_x^2}.$$
Formulas (\ref{tauu}) are trivial at the points, where $|\omega_ x|=0$.

If otherwise, let us find the eigenvalues of $g^{-1}$. The identity
\begin{equation}g^{-1}\xi=\xi-\frac{\omega_x\times\omega_x}{1+\omega_x^2}\xi=
\xi-\frac{(\omega_x^T,\xi)}{1+\omega_x^2}\omega_x,\quad \xi\in\mathbb{R}^{n},\label{evg}\end{equation}
makes obvious that $\xi_1=\omega_x^T$ is the eigenvector corresponding to the first eigenvalue $\lambda_1=1/(1+\omega_x^2)$. The other eigenvectors are orthogonal to $\omega_x^T$ and it follows from the line (\ref{evg}) that $\lambda_2=\lambda_3=\ldots=\lambda_{n}=1$.

Let the eigenvectors of $g^{-1}$ be the columns of matrix $C\in{\rm O(n)}$ and $C_1=\omega_x^T/|\omega_x|$. Then we have
$$g^{-1}=C\Lambda C^T,\quad\tau_0=g^{-\frac{1}{2}}=C\Lambda^{\frac{1}{2}}C^T,\quad\Lambda={\rm diag}\left(\frac{1}{1+\omega_x^2},1,1,\ldots,1\right).$$
Hence,
$$(\tau_0)^i_j=\sum_{k=1}^{n}\sqrt{\lambda_k}c^i_kc^j_k=\frac{1}{\sqrt{1+\omega_x^2}}c^i_1c^j_1+\sum_{k=2}^{n}c^i_kc^j_k
=\delta^i_j+\omega_i\omega_j\left(\frac{1}{\omega^2_x\sqrt{1+\omega_x^2}}-\frac{1}{\omega^2_x}\right),$$
and the rearranged righthand side of above brings out (\ref{tauu}).

Finally, we compute the curvature matrix by the formula (\ref{cms}):
$$X_{kl}=\left(\matrix{{\mathbf 0}\cr\omega_{kl}\cr}\right),\quad \mathbf{n}^+=\frac{1}{\sqrt{1+\omega_{x}^2}}\left(\matrix{-\omega^T_{x}\cr 1 \cr}\right),\quad\mathcal{K}_{ij}=
\frac{\omega_{kl}\tau^k_i\tau^l_j}{\sqrt{1+\omega_x^2}}.$$
\end{proof}
Every hypersurface admits a local graph-parametrization (\ref{yav}) and formulas (\ref{tauu}), (\ref{krr}) significantly decrease amount of calculations in applications. Indeed, consider the problem of computing matrices (\ref{krr}) at some point $M_0\in \Gamma^n$. Since the curvature matrices are geometric invariants, we are free to appoint $M_0=\mathbf 0$ and $\omega(0)=0$, $\omega_ x(0)=\bold 0$ in (\ref{yav}). If so, it follows from (\ref{tauu}), (\ref{krr}) that
\begin{equation}\tau_0(M_0)=I,\quad \mathcal{K}^{\pm}[\Gamma^n](M_0)=B^T\omega_ {xx}(\bold 0)B. \label{kr0}\end{equation}
To show our formulas at work consider two simple examples.

\begin{enumerate}
\item Let $\Gamma^n$ be a sphere $\mathcal S^n_R\subset\mathbb R^{n+1}$, $M_0\in\mathcal S^n_R$. To make use of (\ref{kr0}) appoint $M_0=\mathbf 0$, $x^{n+1}=\omega(x)=R-\sqrt{R^2-x^2}$. It is obvious that $\omega_{xx}(\bold 0)={I}/{R}$ and presentation (\ref{kr0}) brings out
\begin{equation}\tau_0(M_0)=I,\quad\mathcal K[\mathcal S^n_R](M_0)=\frac{1}{R}I.\label{sph}\end{equation}
    Since $M_0$ is an arbitrary point, we arrive to conclusion that  $\mathcal K[\mathcal S^n_R]={I}/{R}$.
\item
 Let $\Gamma^n\subset\mathbb R^{n+1}$ be the upper part of a hyperboloid (unclosed surface):
\begin{equation} \Gamma^n=\left\{|x|>R>0,\;x^{n+1}=\omega(x)=\sqrt{x^2-R^2}\right\}.\label{Hgr}\end{equation}
It follows from (\ref{yav}), (\ref{tauu}), (\ref{krr}) that
$$\omega_ x=\frac{x^T}{\omega},\quad\omega_{xx}=\frac{1}{\omega}\left(I-\frac{x\times x}{\omega^2}\right),$$
$$\tau_0[\Gamma^n]=I-\frac{x\times x}{\sqrt{\omega^2+x^2}(\omega+\sqrt{\omega^2+x^2})},$$
$$\mathcal K_0[\Gamma^n]=\frac{\pm1}{\sqrt{\omega^2+x^2}}\tau^T_0\left(I-\frac{x\times x}{\omega^2}\right)\tau_0.$$
Making use of the identity $(x\times x)^2=|x|^2(x\times x)$, we derive from the above formula
\begin{equation} \mathcal K_0[\Gamma^n]=
\frac{\pm1}{\sqrt{2x^2-R^2}}\left(I-\frac{2x\times x}{2x^2-R^2}\right)\;\;\Rightarrow\;\;\mathcal K[\Gamma^n]=B\mathcal K_0[\Gamma^n]B^T,\; B\in{\rm O}(n).\label{gip}\end{equation}
The choice of the sign in (\ref{gip}) will be made in the Section 5.
\end{enumerate}

\subsection{Normal and principal curvatures}
The classic approach to the definition of normal and principal curvatures involves the first and the second quadratic forms of a hypersurface. Namely, let $X[\Gamma^n](\theta)$ be some local parametrization of $\Gamma^n$, $M\in\Gamma^n$. Let the normal $\mathbf n^+[\Gamma^n]$ fixes an orientation of $\Gamma^n$ in the ambient space. For an arbitrarily fixed vector $d\theta\in\mathbb R^n$ the differential $dX(M)=X_k(M) d\theta^k$ belongs to the tangent $n$-plane (\ref{tP}).

Then by definition, the scalar products
$$(dX,dX)=(g[\Gamma^n]d\theta,d\theta),\quad g=X_{\theta}^TX_{\theta},$$
\begin{equation}(d^2X,\mathbf n^+)=(b[\Gamma^n]d\theta,d\theta),\quad b=(X_{\theta\theta},\mathbf n^+)\label{12f}\end{equation}
are the first and the second quadratic forms respectively.
The forms (\ref{12f}) are geometric invariants in the sense of Definition 2.4, while the matrices $g$, $b\in{\rm Sym}(n)$ are not.

The classic absolute invariants of hypersurfaces are often formulated via ratio
\begin{equation}J(M,d\theta)=\frac{(d^2X,\mathbf n^+)}{(dX,dX)}=\frac{(bd\theta,d\theta)}{(gd\theta,d\theta)},\quad d\theta\in\mathbb R^n,\label{normkr1}\end{equation}
and read as follows:
\begin{enumerate}
\item[(i)] $J(M,d\theta)$ is the {\it normal curvature} of $\Gamma^n$ in the direction $dX(M)=X_k(M)d\theta^k$, that is the curvature of the intersection of $\Gamma^n$ with the $2$-dimensional plane containing $\mathbf n^+[\Gamma^n](M)$ and $dX(M)$;

\item[(ii)] the {\it principal curvatures $\{\kappa_i\}_{1}^n$} of $\Gamma^n$ at $M$ are the stationary values of ratio (\ref{normkr1}), i.e., the eigenvalues of the matrix $g^{-1}b$;

\item[(iii)] solutions $d\theta$ to equations $J(M,d\theta)=\kappa_i$, i.e., the eigenvectors of $g^{-1}b$ are the coordinates of the {\it principal directions} in the basis $\{X_k\}_{1}^n$.
\end{enumerate}

Speaking about direction in $(i)$ either principal directions in $(iii)$ one should understand them in the broad sense. In fact, only collinearity of these {\it directions} with the relevant straight lines is a geometric invariant.

Now we reformulate (\ref{normkr1}) and $(i)$ -- $(iii)$ in our terms of curvature matrix and $1$-direction (see Definition 4.1, Definition 3.5 and example 1 after it).

Denote by
\begin{equation}\mathbf k[\Gamma^n](M,\eta)=(\mathcal K[\Gamma^n](M)\eta,\eta), \quad \eta\in\mathbb R^n, \quad |\eta|=1,\label{Ceta}\end{equation}
the quadratic form generated by the curvature matrix $\mathcal K[\Gamma^n]\in{\rm Sym}(n)$.

Consider the geometric invariant description (\ref{IL}) of the tangent $n$-plane. For an arbitrary vector $\eta$ from (\ref{Ceta}) the tangent vector $l[\Gamma^n](M)=X_{(j)}[\Gamma^n](M)\eta^j$ fixes some $1$-direction $\mathcal L^1_\eta[\Gamma^n](M)$ with $\eta$ as cartesian allocator.

Denote by $\Gamma^1_\eta$ the intersection of $\Gamma^n$ with the $2$-dimensional plane $\mathcal P^2_l(M)$ which spans $\mathbf n^+[\Gamma^n](M)$ and $l[\Gamma^n](M)$. According to (\ref{krd}) we denote by $\kappa[\Gamma^1_\eta]$ the curvature of the plane curve $\Gamma^1_\eta$.

\begin{Th}
Let $M\in\Gamma^n$, $\Gamma^n$ be a $C^k$-smooth hypersurface, $k\geqslant2$.

Then
\begin{enumerate}
\item[(i)] $\mathbf k[\Gamma^n](M,\eta)$ is the normal curvature of $\Gamma^n$ at $M$ in the $1$-direction with allocator $\eta$, i.e., $\mathbf k[\Gamma^n](M,\eta)=\kappa[\Gamma^1_\eta](M)$;

\item[(ii)] the principal curvatures $\{\kappa_i\}_{1}^n$ of $\Gamma^n$ at $M$ are the eigenvalues of the curvature matrix $\mathcal K[\Gamma^n](M)$;

\item[(iii)] the eigenvectors of $\mathcal K[\Gamma^n](M)$ are the cartesian allocators of the principal $1$-directions.
\end{enumerate}
\end{Th}
\begin{proof}
Let $\eta\in\mathbb R^n$, $|\eta|=1$. Put $d\theta=\tau\eta$ with matrix $\tau=\tau[\Gamma^n](M)$ from (\ref{tau3}).
It follows from (\ref{tau1}), (\ref{cms}), (\ref{12f}), (\ref{normkr1}), (\ref{Ceta}) that
$$J(M,d\theta)=\frac{(bd\theta,d\theta)}{(gd\theta,d\theta)}=\frac{\left((X_{\theta\theta},\mathbf n^+)\tau\eta,\tau\eta\right)}{\left((\tau\tau^{T})^{-1}\tau\eta,\tau\eta\right)}=
(\mathcal K[\Gamma^n](M)\eta,\eta)=\mathbf k[\Gamma^n](M,\eta).$$
Therefore, $\mathbf k[\Gamma^n](M,\eta)$ is the curvature of the intersection of $\Gamma^n$ with the $2$-dimensional plane containing $\mathbf n^+[\Gamma^n](M)$ and $dX(M)=X_kd\theta^k$. This intersection coincides with $\Gamma^1_\eta$, since in view of (\ref{j})
$$dX(M)=X_k\tau^k_j\eta^j=X_{(j)}\eta^j=l[\Gamma^n](M).$$
This validates $(i)$. Assertions $(ii)$, $(iii)$ are straightforward consequences of $(i)$.
\end{proof}

Since the curvature of the plane curve $\Gamma^1_\eta$ may be interpreted as its curvature matrix, see (\ref{krd}), equality $(i)$ brings out the following relation
$$\mathcal K[\Gamma^1_\eta](M)=\eta^T\mathcal K[\Gamma^n](M)\eta,$$
which admits a natural extension from $q=1$ to $q=2,\dots, n$.

Consider some $(n\times q)$-matrix $\eta$, $\eta^T\eta=I\in{\rm Sym}(q)$. Let $M\in\Gamma^n$ and $\eta$ be a cartesian allocator of some $q$-direction $\mathcal L^q_\eta[\Gamma^n](M)$, see Definition 3.5. Consider a $(q+1)$-dimensional plane $\mathcal P^{q+1}_\eta[\Gamma^n](M)$ such that
$$\{\mathbf n^+[\Gamma^n](M), \mathcal L^q_\eta[\Gamma^n](M)\}\subset\mathcal P^{q+1}_\eta[\Gamma^n](M).$$

\begin{Def} Let $1\leqslant q\leqslant n$.
We say that $\mathcal P^{q+1}_\eta[\Gamma^n](M)$ is a {\it normal $(q+1)$-sectional plane} and the surface
\begin{equation}\Gamma^q_\eta=\Gamma^n\cap\mathcal P^{q+1}_\eta[\Gamma^n](M)\label{nql}\end{equation}
is a {\it normal $q$-section} of $\Gamma^n$ at $M$.
\end{Def}

Notice that any normal $(q+1)$-sectional plane may be considered as $\mathbb R^{q+1}$. Then surface (\ref{nql}) turns into a $q$-dimensional hypersurface in $\mathbb R^{q+1}$ and Definition 4.1 provides the curvature matrix $\mathcal K[\Gamma^{q}_\eta]\in{\rm Sym(q)}$ at all points of $\Gamma^{q}_\eta$. However, in this paper only points with $\mathbf n^+[\Gamma^q_\eta](M)=\mathbf n^+[\Gamma^n](M)$ are of interest.

\begin{Def}
Let $\Gamma^{q}_\eta$ be a normal $q$-section of $\Gamma^n$ at a point $M$. We say that the matrix $\mathcal K[\Gamma^q_\eta](M)$ is the {\it normal $q$-sectional curvature matrix} of $\Gamma^{n}$ at $M$ in the $q$-direction with allocator $\eta$.
\end{Def}

The notion of $q$-connection on a surface via subsurface, see Definition 3.7, allows to write out all normal $q$-sectional curvature matrices in terms of $\mathcal K[\Gamma^{n}](M)$.

\begin{Th}
Let $\Gamma^q_\eta$ be some normal $q$-section of $\Gamma^n$ at $M$, $\eta$ be some cartesian allocator of corresponding $q$-direction $\mathcal L^q_\eta[\Gamma^n](M)$ and $\zeta$ be the $q$-connection on $\Gamma^n$ via $\Gamma^q_\eta$ defined by (\ref{dz}). Then
\begin{equation}B\mathcal K[\Gamma^{q}_\eta](M)B^T=\eta^T\mathcal K[\Gamma^n](M)\eta,\quad B=\eta^T\zeta(M)\in{\rm O}(q).\label{BKq}\end{equation}
\end{Th}
\begin{proof}
Consider some local parametrization $X[\Gamma^n](\theta)$, $\theta\in\Theta\subset\mathbb R^{n}$, in a neighborhood of $M\in\Gamma^{n}\subset\mathbb R^{n+1}$. If the hypersurface $\Gamma^n$ is $C^k$-smooth, its normal $q$-section is $C^k$-smooth embedding at least locally, see Definition 3.3. So there exists a $C^k$-smooth mapping $\Xi\rightarrow\Theta$ such that $\det\theta_{\xi}^T\theta_{\xi}\neq0$ and
$$
X[\Gamma^q_\eta](\xi)=X[\Gamma^n](\theta(\xi)),\quad\theta=\theta(\xi),\quad \xi\in\Xi\subset\mathbb R^{q},
$$
in some $\mathbb R^{q+1}$-vicinity of $M$.

Similar to (\ref{dz}) double invariant differentiation of this identity brings out the line
\begin{equation}X_{(\xi)(\xi)}[\Gamma^q_\eta](\xi)=\left(\zeta^TX_{(\theta)(\theta)}[\Gamma^n]\zeta\right)(\xi)+
\left(X_{(\theta)}[\Gamma^n]\zeta_{(\xi)}\right)(\xi).\label{dz2}\end{equation}
We have $\mathbf n^+[\Gamma^q_\eta](M)=\mathbf n^+[\Gamma^n](M)\in\mathcal P^{q+1}_\eta[\Gamma^n](M)$ and Definition 4.1 with formula (\ref{cm}) serves both $\Gamma^n$, $\Gamma^q_\eta$. Therefore, equality (\ref{dz2}) restricted to $M$ leads to
\begin{equation}\mathcal K[\Gamma^q_\eta](M)=\left(\zeta^T\mathcal K[\Gamma^n]\zeta\right)(M),\quad \zeta=\tau^{-1}[\Gamma^n]\theta_{\xi}\tau[\Gamma^q_\eta].\label{dzK}\end{equation}
But it follows from Lemma 3.8 that $\zeta(M)$ is one of cartesian allocators of $\mathcal L^q_\eta[\Gamma^n](M)$. Due to remark 3.6 there exists a matrix $B\in{\rm O}(q)$ such that $\zeta(M)=\eta B$ and $B=\eta^T\zeta(M)$. In cooperation with (\ref{dzK}) it completes the proof of relation (\ref{BKq}).
\end{proof}

We see the inequality (\ref{dzK}) is a new result in differential geometry.

Keeping in mind (\ref{BKq}) we can say that the normal $q$-sectional curvature matrix described in Definition 4.6 is a generalization of the normal curvature (\ref{Ceta}), see Theorem 4.4, $(i)$.
Definition 4.6 as well opens a way to notions of {\it $q$-sectional principal curvatures} of $\Gamma^n$ at $M$. However, it is of no use in this paper.

\section{On $p$-curvatures}

\subsection{Cone of $m$-positive matrices}

The symmetry of the curvature matrix is crucial for application of L. G{\aa}rding algebraic theory \cite{G59}
in geometric developments. To describe this connection we start with a brief algebraic survey.

Denote by $\lambda(S)=(\lambda_1,\lambda_2,\ldots,\lambda_n)$ the set of the eigenvalues of some matrix $S\in{\rm Sym}(n)$. Following L. G{\aa}rding, we substitute the classic characteristic polynomial of $S$ by $\det(S+tI)$, $t\in\mathbb R$, and deduce the presentation
\begin{equation}\det(S+tI)=\sum_{p=0}^{n}T_p(S)\cdot t^{n-p}=\prod_{i=1}^n(t+\lambda_i(S)),\quad t, \lambda_i\in\mathbb R\label{raz}.\end{equation}
Here $T_p(S)$ is the sum of all $p$-s order principal minors of $\det S$ by Definition 3.2. It follows from (\ref{raz}) that
\begin{equation}T_p(S)=\sigma_p(\lambda(S))=\sum_{\{i_1<i_2<\ldots< i_p\}}\lambda_{i_1}\lambda_{i_2}\ldots \lambda_{i_p}, \quad 1\leqslant p\leqslant n.\label{sim}\end{equation}
It is of common knowledge that the elementary symmetric functions $\{\sigma_p(\lambda(S))\}_1^n$ satisfy the identity $\sigma_p(\lambda(S))= \sigma_p(\lambda(B^TSB))$. Therefore, relation (\ref{sim}) means that $p$-traces of $S$ are {\it orthogonal invariant} in the sense
\begin{equation}T_p(B^TSB)=T_p(S),\quad B\in{\rm O}(n).\label{ort}\end{equation}

\begin{Def}
Let $0\leqslant m\leqslant n$. We say that $S\in{\rm Sym}(n)$ is $m$-{\it positive} matrix if it belongs to the following cone
\begin{equation} K_m=K_m(n)=\{S\in {\rm Sym}(n):\; T_p(S)>0,\; p=0,1,\ldots,m\}.\label{Km}\end{equation}
\end{Def}
Since $T_0(S)\equiv1$, any symmetric matrix is $0$-positive and $K_0={\rm Sym}(n)$.
Notice that all coefficients of polynomial (\ref{raz}) are positive if and only if all its roots are negative. But the roots of (\ref{raz}) are the eigenvalues of $S$ with opposite sign. Hence, $K_n$ is exactly the cone of positive definite symmetric matrices.

Thus, Definition 5.1 generates the following stratification in the space of symmetric matrices:
\begin{equation}{\rm Sym}(n)=K_0\supset K_1\supset K_2\supset\dots\supset K_n.\label{sist}\end{equation}

The cones (\ref{Km}) were introduced  by N.M.Ivochkina  as {\it the cones of stability},
\cite{I83}. The term ``$m$-positive matrix'' for elements from $K_m$ was recently introduced in \cite{IYP12} and \cite{I12}.

It turned out that $K_m$ are the examples to L.G{\aa}rding algebraic theory of $a$-hyperbolic polynomials and cones, \cite{G59}.

\begin{remark}
In the paper \cite{G59} L.G{\aa}rding introduced the notion of $a$-hyperbolic polynomial. This is a homogeneous polynomial $P_m(s)$, $s\in\mathbb R^N$, such that the polynomial
$$p(t)=P_m(s+ta)=P_m(a)\prod_{i=1}^m\left(t+\lambda_i(a,s)\right),\quad a\in\mathbb R^N,\quad t\in\mathbb R,$$
has real roots $\{-\lambda_i(a,s)\}_1^m$ for all $s\in\mathbb R^N$.
Section 2, \cite{G59}, contains the central notion in this theory:
let $C(P_m,a)$ be the set of all $s$ such that $P_m(s+ta)\neq0$ when $t\geqslant 0$. The set $C(P_m,a)$ is now called the G{\aa}rding cone corresponding to $P_m$.
It is close to obvious (see \cite{FB16}) that polynomials $P_m=T_m(S)$, $m=1,\dots,n$, are $I$-hyperbolic in $\mathbb R^N$ with $N={n(n+1)}/{2}$ and $C(T_m,I)=K_m$. Below we reformulate some basic properties of G{\aa}rding cones in terms of the matrix cones $K_m$.

The significance of L.G{\aa}rding theory in FNPDE was firstly revealed in the paper
\cite{CNS85}, p.268, Section 1. The paper
\cite{IYP12}
contains a description of algebraic aspects of L.G{\aa}rding theory and renovated proofs of basic theorems.
The paper
\cite{FB16}
contains a popular review of L.G{\aa}rding theory and its relations to $T_m$, $K_m$ and FNPDE.
\end{remark}

\begin{lemma}
Let $S^0\in{\rm Sym}(n)$ be some $m$-positive matrix. The cone $K_m$ coincides with a connected component of the set $\{S\in {\rm Sym}(n): T_m(S)>0\}$, the component containing $S_0$.
\end{lemma}
The simple proof of Lemma 5.2 can be found in the paper \cite{IF15smfn} (Lemma 2.2).

Enumerate now some principal properties of $T_m$ in $K_m$:
\begin{enumerate}
\item The cone $K_m$ is convex in ${\rm Sym}(n)$, the function $F_m=T_m^{\frac{1}{m}}$ is concave in $K_m$ and $1$-homogeneity of $F_m$ carries out the inequality
\begin{equation}
F_m(S)\leqslant F^{ij}_m(\tilde{S})s_{ij},\quad S,\tilde{S}\in K_m,\quad F_m^{ij}(\tilde S)=\frac{\partial F_m(\tilde S)}{\partial \tilde s_{ij}}.\label{con}
\end{equation}

\item There are the inequalities
\begin{equation} F_{p-1}(S)>F_p(S),\quad S\in K_m,\quad  1\leqslant p\leqslant m.\label{macl}\end{equation}

\item The function $T_m$ is positive monotone in the sense
\begin{equation} T_m(S+\tilde{S})>T_m(S),\quad S\in K_m,\quad\tilde{S}\in\bar{K}_m,\,\tilde{S}\ne\bf0.\label{mon}\end{equation}
This implies that inclusions $S\in K_m$, $\tilde{S}\in\bar{K}_m$ guarantee $(S+\tilde{S})\in K_m$.

\item The skew symmetry of $T_m$ provides the inequality
\begin{equation}T_m(S+\xi\times\xi)=T_m(S)+T_m^{ij}(S)\xi_i\xi_j,\quad T_m^{ij}(S)=\frac{\partial T_m(S)}{\partial s_{ij}},\quad \xi\in\mathbb R^n.\label{kos}\end{equation}
Since $\xi\times\xi\in\bar{K}_n\subset\bar{K}_m$, (\ref{mon}) and (\ref{kos}) validate the inequality
\begin{equation}T_m^{ij}(S)\xi_i\xi_j>0,\quad S\in K_m,\quad \xi\ne\bf0.\label{ellip}\end{equation}
\end{enumerate}

In the paper \cite{F14prep}, \cite{IF15smfn} N. V. Filimonenkova extended Sylvester criterion for positive definite matrices onto $m$-positive matrices.
\begin{lemma}{\bf(Sylvester criterion)}
Let $S\in {\rm Sym}(n)$, $1\leqslant m\leqslant n$.

(i) Fix some index $i$, $1\leqslant i\leqslant n$. Denote by $S^{\langle i\rangle}\in {\rm Sym}(n-1)$ the matrix derived from $S$ by crossing out the row and the column numbered by $i$. Then
$$S\in K_m(n)\quad\Leftrightarrow\quad T_{m}(S)>0, \; S^{\langle i\rangle}\in K_{m-1}(n-1).$$

(ii) Fix some collection of different indexes $1\leqslant i_1,i_2,\ldots,i_{m-1}\leqslant n$. Denote by $S^{\langle i_1,i_2,\ldots,i_k\rangle}\in {\rm Sym}(n-k)$ the matrix derived from $S$ by crossing out $k$ rows and $k$ columns with corresponding numbers.
For $S$ to be $m$-positive it is necessary and sufficient to have
$$T_{m}(S)>0,\;T_{m-1}(S^{\langle i_1\rangle})>0,\;T_{m-2}(S^{\langle i_1,i_2\rangle})>0,\ldots,T_{1}(S^{\langle i_1,i_2,\ldots,i_{m-1}\rangle})>0.$$
\end{lemma}

\begin{remark} Since indexes $i_1,i_2,\ldots,i_{m-1}$ are arbitrary, it follows that if $S\in K_m$, then $S$ has at least $m$ positive eigenvalues (the converse implication is only true for $m=n$).
\end{remark}

\subsection{On $m$-convex hypersurfaces}

The algebraic notions described in the previous Subsection give rise to new geometric invariants. We start with the notion of $p$-curvature of a hypersurface. Below we deal with $C^k$-smooth oriented hypersurfaces $\Gamma^n\subset\mathbb R^{n+1}$, $k\geqslant 2$.
\begin{Def}
The functions
\begin{equation}\mathbf k_p[\Gamma^n](M)=T_p(\mathcal K[\Gamma^n])(M),\quad M\in\Gamma^n,\quad p=1,\dots,n,\label{p-kr}\end{equation}
are the $p$-{\it curvatures} of $\Gamma^{n}$. By definition $\mathbf k_0\equiv 1$.
\end{Def}

The following proposition is a straightforward consequence of Definition 5.6, remark 4.2, relations (\ref{sim}) and Theorem 4.4, $(ii)$.
\begin{lemma}
Let $\Gamma^n$ be an oriented $C^k$-smooth hypersurface, $k\geqslant 2$. Then $p$-curvatures of $\Gamma^{n}$ are $C^{k-2}$-smooth functions of $M\in\Gamma^n$, absolute geometric invariants of $\Gamma^n$ and
\begin{equation}\mathbf k_p[\Gamma^n]=\sigma_p(\kappa_{1},\kappa_{2},\ldots ,\kappa_{{n}})=\sum_{i_1<i_2<\ldots< i_p}\kappa_{i_1}\kappa_{i_2}\ldots \kappa_{i_p},\quad 1\leqslant p\leqslant n, \label{simkr}\end{equation}
where $\{\kappa_i[\Gamma^n]\}_1^{n}$  is the set of the principal curvatures.
\end{lemma}

Notice that definitions (\ref{p-kr}) and (\ref{simkr}) are equivalent, but the first is preferable to apply in calculations. It makes obvious the smoothness of $p$-curvatures.

It follows from (\ref{simkr}) that the mean curvature and the Gauss curvature are included into Definition 5.6 as $1$-curvature and $n$-curvature. The geometric invariant (\ref{bc}) is the $(m-1)$-curvature of the hypersurface $\partial\Omega\subset\mathbb R^n$ in our terminology.

Definitions 4.1, 5.1 induce the following classification of smooth hypersurfaces.
\begin{Def}
Let $0\leqslant m\leqslant n$, $\Gamma^{n}$ is a $C^k$-smooth oriented hypersurface, $k\geqslant2$. The hypersurface $\Gamma^{n}$ is $m$-{\it convex} at a point $M\in\Gamma^{n}$ if its curvature matrix $\mathcal K[\Gamma](M)$ is $m$-positive and $\Gamma^{n}$ is $m$-convex if all its points are the points of $m$-convexity.
\end{Def}

Denote by ${\mathbf{K}}_m$ the set of all $m$-convex in $\mathbb{R}^{n+1}$ hypersurfaces. The lines (\ref{Km}), (\ref{p-kr}) brings out a constructive description of this set:
\begin{equation}{\mathbf{K}}_m=\{\Gamma^n\subset \mathbb{R}^{n+1}:\; \mathbf k_p[\Gamma^n](M)>0,\;p=0,1,\ldots,m,\;M\in\Gamma^n\}.\label{ccg}\end{equation}
The following proposition is a straightforward consequence of relations (\ref{tau1}), (\ref{cms}), (\ref{p-kr}).
\begin{lemma}
Let $X[\hat\Gamma^n]=\alpha X[\Gamma^n]$, $\alpha>0$. Then the equalities
$$\mathbf k_p[\hat\Gamma^n]=\frac{1}{\alpha^p}\mathbf k_p[\Gamma^n],\quad0\leqslant p\leqslant n,$$
are true.
\end{lemma}
Cooperation of Lemma 5.9 with (\ref{ccg}) gives a reason to speak about cones of $m$-convex hypersurfaces and to expose a geometric replica of (\ref{sist}):
$$\mathbf{K}_0\supset \mathbf{K}_1\supset \mathbf{K}_2\supset\dots\supset \mathbf{K}_n.$$
Notice that an arbitrary $C^2$-smooth hypersurface is $0$-convex by definition.

It follows from Remark 5.5 that any $m$-convex hypersurface has at least $m$ positive principal curvatures and any $n$-convex hypersurface is strongly convex in the classic sense.

However, Definition 5.8 with $m=n$ is more restrictive than textbook definitions of the strong convexity. Indeed, in classic geometry strongly convex hypersurfaces were introduced as the boundaries of strongly convex domains and might be just Lipschitz continuous. Also, smooth strongly convex hypersurfaces may contain isolated points of the vanishing Gauss curvature, while Definition 5.8 rules out such points when $m=n$. Moreover, a smooth strongly convex hypersurface may contain isolated planar points, where all principal curvatures vanish. If so, a hypersurface is not $1$-convex, let alone $n$-convexity.

Notice that $p$-convexity of boundaries, $p\geqslant 1$, regulates existence and non-existence theorems in FNPDE, while $0$-convexity serves fine the linear theory of second-order partial differential equations. Hence, $p$-convexity of hypersurfaces is vital in applications.

\begin{remark}
In order to remove the problem of orientation from Definition 5.8, it is reasonable to keep to the rule:
choose an arbitrary point $M\in\Gamma ^n$ and direct $\mathbf n^+[\Gamma ^n](M)$ such way that $\mathbf k_1[\Gamma ^n](M)>0$; over $\Gamma ^n$ we construct the field of normals $\{\mathbf n^+[\Gamma ^n]\}$ consistent with $\mathbf n^+[\Gamma ^n](M)$.
The above rule does not work if $\mathbf k_1[\Gamma ^n](M)=0$, wich just means that our hypersurface is $0$-convex at the point $M$. Then we choose another point. Notice that the above recommendation agrees with the preference of interior normals for closed $C^k$-smooth hypersurfaces.
\end{remark}

Consider the examples from Subsection 4.2.

\begin{enumerate}
\item Any sphere $\mathcal {S}^n_R\subset\mathbb R^{n+1}$ is $n$-convex hypersurface. Indeed, the $p$-traces of its curvature matrix, see (\ref{sph}), are
$$\mathbf k_p[\mathcal S^n_R]=\frac{C_n^p}{R^p}>0,\quad p=1,\dots,n,$$
and constructive definition (\ref{ccg}) confirms the $n$-convexity of $\mathcal S^n_R$.
\item Turn to the hyperboloid
\begin{equation}\mathcal {H}^n_R=\{|x|\geqslant R>0, x^2-(x^{n+1})^2=R^2\}\subset\mathbb R^{n+1}.\label{Hp}\end{equation}
Alike the sphere it is $C^\infty$-smooth and symmetric over the plane $x^{n+1}=0$. Therefore, it is sufficient to examine the graph-presentation (\ref{Hgr}) and extend conclusions to (\ref{Hp}) by the symmetry.

We can take $B=I$ in (\ref{gip}) and Remark 5.10 assigns the choice $\mathcal K[\mathcal {H}^n_R]=+\mathcal K_0$, since
$$\mathbf k_1[\mathcal {H}^n_R]=T_1(+\mathcal K_0)=\frac{1}{\sqrt{2x^2-R^2}}\left(n-\frac{2x^2}{2x^2-R^2}\right)=$$
\begin{equation}=\frac{1}{\sqrt{(2x^2-R^2)^3}}(2(n-1)x^2-nR^2)>0,\quad |x|>R,\quad n>1.\label{HT1}\end{equation}

Using the identity (\ref{kos}), we compute
\begin{equation}\mathbf k_m[\mathcal {H}^n_R]=T_m(+\mathcal K_0)=\frac{C_{n-1}^{m-1}}{\sqrt{(2x^2-R^2)^m}}\left(\frac{n}{m}-\frac{2x^2}{2x^2-R^2}\right), \quad m=1,\dots,n.
\label{KHtr}\end{equation}
The following lemma is a simple consequence of (\ref{ccg}) and (\ref{KHtr}).
\begin{lemma}
Let $n>1$, $1\leqslant m<n$ and $\mathcal H^n_R$ be the hyperboloid (\ref{Hp}). Then the hypersurface
\begin{equation}\mathcal H^n_R\cap\left\{x^2>\frac{n}{2(n-m)}R^2\right\}\label{Hpx}\end{equation}
is $m$-convex. In particular, the hypersurface $\mathcal H^n_R$ upgrades to $(n-1)$-convexity when $x^2>nR^2/2$.
\end{lemma}
\begin{proof}Indeed, the hypersurface (\ref{Hpx}) has positive $m$-curvature (\ref{KHtr}). In addition, if $x$ satisfies inequality (\ref{Hpx}), the similar inequalities with $1\leqslant p<m$ instead of $m$ are also keep. Hence, our hypersurface is $m$-convex via definition (\ref{ccg}).
\end{proof}

There are some specialities when $n=2$ and $n=1$.

Namely, $\mathcal H^2_R$ is just $0$-convex, although upgrades to $1$-convex at the moment $|x|>R$. In total it may be qualified as a non-negative mean curvature hypersurface.

Lemma 5.11 excludes $n=1$, since in Remark 5.10 assigns the choice $\mathcal K[\mathcal {H}^1_R]=-\mathcal K_0$ from (\ref{gip}). Then formula (\ref{HT1}) with $+\mathcal K_0$ substituted on $-\mathcal K_0$ carries out the $1$-curvature of the classic equilateral hyperbola $x^2-y^2=R^2$:
$$\kappa[\mathcal H_R^1]=\frac{R^2}{\sqrt{(x^2+y^2)^3}}.$$
 So, in our classification the classic hyperbola is a $1$-convex curve, that is a curve of the positive curvature $\kappa[\mathcal H_R^1]$ and should be strongly convex in the classic sense.
\end{enumerate}

Other examples can be found in the paper \cite{FB17} dedicated to the $m$-convexity of multidimensional paraboloids and hyperboloids.

Eventually we state one more test of $m$-convexity.
\begin{lemma}
A $C^k$-smooth, $k\geqslant 2$,  hypersurface  $\Gamma^{n}$ is $m$-convex if and only if the following conditions hold:
\begin{enumerate}
\item[(i)] there exists a point $M_0\in\Gamma^{n}$ such that $\Gamma^{n}$ is $m$-convex at $M_0$;

\item[(ii)]  $\mathbf k_m[\Gamma^{n}](M)>0$ at all points $M\in\Gamma^{n}$.
\end{enumerate}
\end{lemma}
Actually Lemma 5.12 is just a geometric replica of Lemma 5.3. It looks particulary impressive for closed hypersurfaces.
\begin{corollary}
A closed $C^k$-smooth, $k\geqslant 2$,  hypersurface  $\Gamma^{n}$ is $m$-convex if and only if $\mathbf k_m[\Gamma^{n}]>0$.
\end{corollary}

\subsection{Sylvester criterion for hypersurfaces}
In this Subsection, using Sylvester criterion for $m$-positive matrices (see Lemma 5.4), we prove a new criterion of $m$-convexity for hypersurfaces.

First we deduce an appropriate modification of Sylvester criterion.
\begin{lemma}
Let $S\in {\rm Sym}(n)$, $1\leqslant m\leqslant n$.
\begin{enumerate}
\item[(i)] Suppose $A$ is some $n\times (n-1)$-matrix such that $A^TA=I\in {\rm Sym}(n-1)$. Then
$$S\in K_m(n)\quad\Leftrightarrow\quad T_m(S)>0,\;A^TSA\in K_{m-1}(n-1).$$

\item[(ii)] Suppose $\{A_k\}_{k=1}^{m-1}$ is some collection of $(n-k+1)\times (n-k)$-matrices such that $A_k^TA_k=I\in {\rm Sym}(n-k)$. Let $S_0=S$, $S_k=A_k^TS_{k-1}A_k\in {\rm Sym}(n-k)$. Then
$$S\in K_m(n)\quad\Leftrightarrow\quad T_{m}(S)>0,\;T_{m-1}(S_1)>0,\;T_{m-2}(S_2)>0,\ldots,T_{1}(S_{m-1})>0.$$
\end{enumerate}
\end{lemma}

\begin{proof}
The operation $A^TSA$ up to some orthogonal transformation is equivalent to crossing out from matrix $S$ a row and a column with some number $i$. Assume without loss of generality that $i=n$. It is easy to check that for any $n\times (n-1)$-matrix $A$ satisfying  $A^TA=I$ there exists a matrix $B\in{\rm O}(n)$ such that
$$(BA)^i_j=\delta^i_j,\; i,j=1,\dots,n-1, \quad (BA)^n_j=0.$$
Denote $\hat S=BSB^T\in{\rm Sym}(n)$. Then $A^TSA=(AB)^T\hat S(AB)=\hat S^{\langle n\rangle}$ (in notations of Lemma 5.4) and item $(i)$ of Lemma 5.14 follows from item $(i)$ of Lemma 5.4 and the orthogonal invariancy of $p$-traces, (\ref{ort}). Namely,
$$S\in K_m(n)\quad\Leftrightarrow\quad \hat S\in K_m(n) \quad\Leftrightarrow\quad T_m(\hat S)=T_m(S)>0,\;\hat S^{\langle n\rangle}=A^TSA\in K_{m-1}(n-1).$$
The item $(ii)$ turns out a result of several iterations of $(i)$.
\end{proof}

The following geometric version of Sylvester criterion is a straightforward consequence of Definition 4.5, Lemmas 5.14, 5.12 and identity (\ref{BKq}).

\begin{Th} Let $1\leqslant m\leqslant n$, $\Gamma^{n}$ is a $C^k$-smooth oriented hypersurface, $k\geqslant2$.
Assume that $\mathbf k_m[\Gamma^n](M)>0$ at all points $M\in\Gamma^n$.
\begin{enumerate}
\item[(i)] A hypersurface $\Gamma^n$ is $m$-convex if and only if there exist a point $M_0\in\Gamma^n$ and at least one normal $(n-1)$-section $\Gamma^{n-1}$ of $\Gamma^n$ at $M_0$ such that $M_0$ is a point of its $(m-1)$-convexity.

\item[(ii)] A hypersurface $\Gamma^n$ is $m$-convex if and only if there exist a point $M_0\in\Gamma^n$ and at least one sequence of normal sections $\Gamma^{n-m+1}\subset\Gamma^{n-m+2}\dots\subset\Gamma^{n-1}\subset\Gamma^n$ at $M_0$ such that
$$\mathbf k_{m-1}[\Gamma^{n-1}](M_0)>0,\;\mathbf k_{m-2}[\Gamma^{n-2}](M_0)>0,\ldots,\mathbf k_{1}[\Gamma^{n-m+1}](M_0)>0.$$
\end{enumerate}
Moreover, for $m$-convex hypersurfaces assertions $(i)$ and $(ii)$ are valid at all points $M_0\in\Gamma^n$ and for any normal sections.
\end{Th}

\begin{corollary}
Let $\mathbf k_m[\Gamma^n](M)>0$ at all points $M\in\Gamma^n$. Suppose there exist a point $M_0\in\Gamma^n$, an integer $p$, $1\leqslant p\leqslant m-1$, and a normal $(m-p)$-section of $\Gamma^n$ at $M_0$ such that the equality
\begin{equation}\mathbf k_{m-p}[\Gamma^{n-p}](M_0)=0\label{nc}\end{equation}
holds true.  Then the hypersuface $\Gamma^n$ has no points of $m$-convexity.
\end{corollary}

Let us remark that for closed $C^k$-smooth hypersurfaces equality (\ref{nc}) is incompatible with the inequality $\mathbf k_m[\Gamma^n](M)>0$, $M\in\Gamma^n$, see Corollary 5.13. Hence, only unclosed smooth hypersurfaces may satisfy the conditions of Corollary 5.16.

\section{Application of new invariants to FNPDE}

\subsection{Cone of $m$-admissible functions}

In this section we deal with functions $u\in C^2(\bar\Omega)$, where $\Omega\subset\mathbb{R}^n$ is a bounded domain, $\partial\Omega\in C^4$. Denote by $u_x$ the gradient, by $u_{xx}$ the Hessian matrix of $u$ respectively. Introduce for $1\leqslant m\leqslant n$ the $m$-Hessian operators $T_m[u]=T_m(u_{xx})$ and set up in $\bar\Omega$ the Dirichlet problem:
\begin{equation}T_m[u]=f^m>0,\quad u\arrowvert _{\partial\Omega}=\varphi.\label{hes}\end{equation}
The $m$-Hessian equations (\ref{hes}) are fully nonlinear when $m>1$. They appeared in 1983 in the paper
\cite{I83} as a natural generalization of well known Monge -- Ampere equation, $m=n$, and were named of {\it Monge -- Ampere type}. But in contrast to $m=n$ there are no solvability conditions of (\ref{hes}) on the set of convex in $\bar\Omega$ functions for $m<n$. The natural sets of the classic solvability of (\ref{hes}) for $m<n$ were described and named as {\it cones of stability} in
\cite{I83}. The crucial step was to introduce the matrix cones $K_m$, see Definition 5.1, and to extend them to $C^2(\bar\Omega)$:
\begin{equation}\mathbb{K}_m(\bar\Omega)=\{u\in C^2(\bar\Omega): T_p[u](x)>0,\;p=0,1,\ldots,m, \;x\in\bar\Omega\}\label{fKm}\end{equation}
For the cones (\ref{fKm}) a functional analog of inclusions (\ref{sist}) is satisfied.

The first attempt of general approach to the theory of FNPDE was performed by L. Caffarelly, L. Nirenberg, J. Spruck in 1985,
\cite{CNS85}.
In particular, there were described new geometric conditions onto the boundary that are necessary for solvability of (\ref{hes}). Also there was introduced the notion of ``admissible function'' for some wide class of FNPDE. We reduce this notion to $m$-Hessian operators in the following definition.

\begin{Def}\label{mdop}
Let $0\leqslant m\leqslant n$. A function $u\in C^2(\bar\Omega)$ is $m$-{\it admissible} at $x\in\bar\Omega$ if the matrix $u_{xx}$ is $m$-positive at this point. A function $u$ is $m$-admissible in $\bar\Omega$ if it is $m$-admissible at all $x\in\bar\Omega$.
\end{Def}

It is obvious that the set of $m$-admissible in $\bar\Omega$ functions coincides with functional cone (\ref{fKm}).

An arbitrary $C^2$-smooth function is $0$-admissible by definition. The graph of any $n$-admissible function is a smooth strictly convex ($n$-convex in our terms) hypersurface in $\mathbb{R}^{n+1}$. However, for $1\leqslant m<n$ the similar assertion is not true.

The following statement is a  simple consequence of Lemma 5.3.

\begin{lemma}\label{com1}
A function $u\in C^2(\bar\Omega)$ is $m$-admissible in $\bar\Omega$ if and only if there exists a point $x_0\in\bar\Omega$ such that $u_{xx}(x_0)\in K_m$ and $T_m[u]>0$ for all points $x\in\bar\Omega$.
\end{lemma}

So, the cones of $m$-admissible functions (\ref{fKm}) as well as the sets of $m$-convex hypersurfaces (\ref{ccg}) are generated by the algebraic cones of $m$-positive matrices (\ref{Km}) and interaction of those is natural but not trivial. An analysis of this interaction has been started in the paper
\cite{IF13}
and here we present some updated samples.

\begin{Th}
Let $\partial\Omega\in C^{4+\alpha}$, $f\in C^{2+\alpha}(\bar\Omega)$, $f\geqslant\nu>0$, $0<\alpha<1$. Assume that $\varphi={\rm C}={\rm const}$.
\begin{enumerate}
\item[(i)] If $\partial\Omega$ is $(m-1)$-convex, i.e.,
\begin{equation}\mathbf k_{m-1}[\partial\Omega]>0,\label{geom}\end{equation}
there exists a unique in $C^2(\bar\Omega)$ solution $u$ to the problem (\ref{hes}) for odd $m$ and there are exactly two solutions $u$, $-u+2{\rm C}$ for even $m$. Moreover, $u$ is $m$-admissible in $\bar\Omega$.
\vskip .1in
\item[(ii)] If there is a point $M_0\in\partial\Omega$ such that $\mathbf k_{m-1}[\partial\Omega](M_0)=0$, there are no solutions in $C^2(\bar\Omega)$ to the problem (\ref{hes}), whatever smooth $f$ has been. This is equivalent to
$$\{u\in C^2(\bar\Omega): \;u\arrowvert_{\partial\Omega}={\rm const},\; T_m[u]>0\}=\varnothing.$$
\end{enumerate}
\end{Th}
Concerning part $(i)$, the $(m-1)$-convexity of $\partial\Omega$ provides the existence and the uniqueness of $m$-admissible solution in presence of arbitrary sufficiently smooth $\varphi$ in (\ref{hes}). For $C^\infty$ data this was established in 1985,
\cite{CNS85}, see Theorem 1.1 in our Section 1. In  2011  N. V. Filimonenkova,
\cite{F11},
investigated the classic and the weak (approximate) solvability of (\ref{hes}) in the cone of $m$-admissible functions. Theorem 1.1 was updated in \cite{F11} and took the following form.
\begin{Th}\label{ex}
Let $\varphi\in C^{l+\alpha}(\partial\Omega)$, $f\in C^{l-2+\alpha}(\bar\Omega)$, $\partial\Omega\in C^{l+\alpha}$, $f>0$ in $\bar\Omega$, $l\geqslant 4$, $0<\alpha<1$. Suppose, in addition, (\ref{geom}) holds.
Then there exists a unique $m$-admissible solution $u\in C^{l+\alpha}(\bar\Omega)$ of (\ref{hes}).
\end{Th}
Notice that the requirement (\ref{geom}) is not necessary for the solvability of (\ref{hes}) with non-constant Dirichlet data.

Due to (\ref{ellip}) the operator $T_m[u]$ is elliptic on functions $u\in\mathbb{K}_m(\bar\Omega)$. More precisely, if $f>0$ in $\bar\Omega$ then $T_m[u]$ is uniformly elliptic on solutions of (\ref{hes}) in presence of a priory estimates in $C^2(\bar\Omega)$.
Construction of these estimates is the most essential part of the proof of existence theorems in the theory of FNPDE.

In linear theory the basis of a priory estimates for solutions to elliptic and parabolic equations is the well known maximum principle, while in the theory of FNPDE it is more natural to exploit the monotonicity of operators via comparison theorems. To formulate one of them we substitute the $m$-traces $T_m$ onto the following $1$-homogeneous operators:
$$F_m[u]=\left(T_m[u]\right)^{\frac{1}{m}},\quad u\in\mathbb K_m(\bar\Omega).$$
Associate with them the $u$-set of the following linear elliptic operators
\begin{equation}\label{LF}L[v;u]=F^{ij}_m[u]v_{ij},\quad F_m^{ij}[u]=\frac{\partial F_m(u_{xx})}{\partial u_{ij}},\quad v\in C^2(\Omega).\end{equation}
\begin{lemma}
Let $u,w\in\mathbb K_m(\Omega)\cap C(\bar\Omega)$, $v\in C^2(\Omega)\cap C(\bar\Omega)$, $\mu>0$. Assume that
\begin{equation}L[v;u]\leqslant\mu,\quad F_m[w]\geqslant\mu.\label{cps1}\end{equation}
Then
\begin{equation}v(x)-w(x)\geqslant\min_{\partial\Omega}(v-w),\quad x\in\Omega.\label{v-w}\end{equation}
\end{lemma}
Indeed,  due to concavity of $F_m$ in $\mathbb K_m$, see (\ref{con}), we have $F_m[w]\leqslant L[w;u]$. Since operators (\ref{LF}) are elliptic, Lemma 6.5 is just a version of the classic maximum principle.

The inequality (\ref{v-w}) reduces the problem of a priori estimates to construction of those at the boundary via $m$-admissible function $w$. We call it an {\it $m$-admissible sub-barrier}. Local boundary sub-barriers can be constructed explicitly. The requirement (\ref{geom}) and the above geometric ideas in fact appeared in this context.

\subsection{Construction of local sub-barriers}

In order to precisely indicate the origin of requirement (\ref{geom}) we introduce the notion of kernel of local sub-barriers. It was first introduced in \cite{IF14fix}.

To begin we associate with a point $M_0\in\partial\Omega$ a domain $\Omega_r$:
\begin{equation}\Omega_r\subset \Omega\cap B_r(M_0),\quad \partial\Omega_r\cap \partial\Omega=\partial\Omega\cap B_r(M_0),\quad 0<r\ll1.\label{sd}\end{equation}

\begin{Def}
We call a function $W$ an $m$-{\it Hessian kernel
of local sub-barriers} at $M_0\in\partial\Omega$
if there is a domain (\ref{sd}) such that $W\in\mathbb K_m(\bar{\Omega}_r)$ and
\begin{equation}W(M_0)=0,\quad W\arrowvert_{\partial\Omega_r}\leqslant 0,\quad W\arrowvert_{\partial\Omega_r\cap\Omega}\leqslant-1.\label{cb}\end{equation}
\end{Def}
Since $W$ is $m$-admissible in $\bar{\Omega}_r$, it has no maximum in ${\Omega}_r$. The following properties of $W$ are an obvious consequence of (\ref{cb}):
\begin{equation}W(x)\leqslant 0,\;x\in\bar{\Omega}_r,\quad (W_x,\mathbf n^+[\partial\Omega])(M_0)\leqslant0, \label{pW}\end{equation}
where $\mathbf n^+[\partial\Omega]$ is the interior normal (see Fig.2).

\begin{figure}
	\center{\includegraphics[width=8cm]{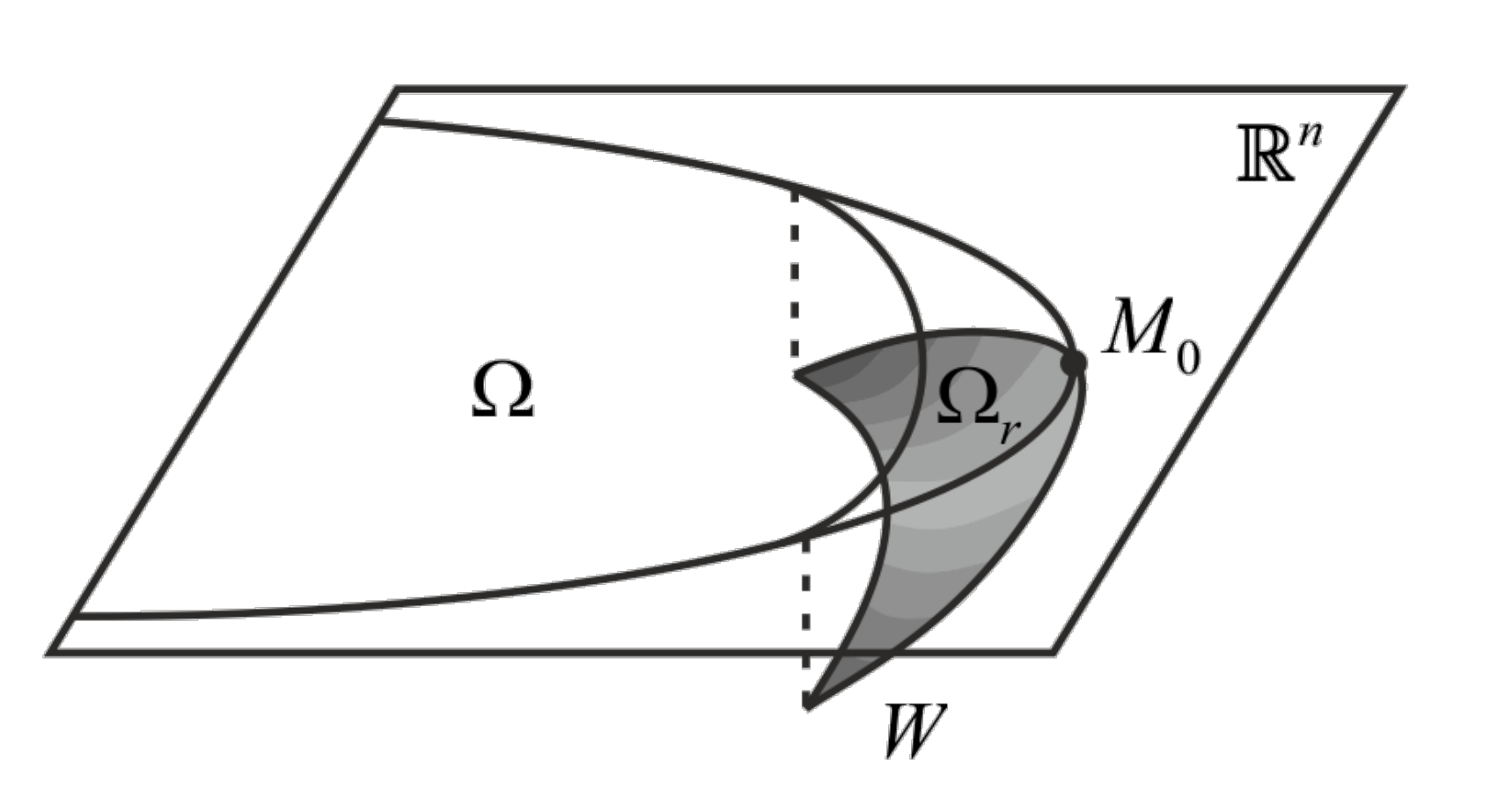} \\ Fig. 2}
\end{figure}

\begin{Th}
Let $\partial\Omega$ be a $C^2$-smooth hypersurface, $M_0\in\partial\Omega$. Assume there exists some function $W$ from Definition 6.6. Then $\partial\Omega$ is $(m-1)$-convex at $M_0$.
\end{Th}
\begin{proof}
Let $\{\mathbf e_1,\mathbf e_2,\dots,\mathbf e_{n}\}$ be a cartesian basis with $\mathbf e_{n}=\mathbf n^+[\partial\Omega](M_0)$. There is a $C^2$-function $\omega=\omega(\tilde{x})$, $|\tilde{x}|\leqslant r$, such that  $\omega(\mathbf 0)=0$,  $\omega_{\tilde x}(\mathbf0)=\mathbf0$ and parametrization of $\partial\Omega\cap B_r(M_0)$ is given by
\begin{equation}X=\left(\matrix{\tilde{x}\cr\omega(\tilde x)\cr}\right),\quad \tilde{x}=\left(\matrix{x^1\cr x^2\cr\ldots \cr x^{n-1}\cr}\right),\quad X(M_0)=\mathbf 0.\label{1}\end{equation}
For $|\tilde{x}|\leqslant r$, Lemma 4.3 ensures the relation
\begin{equation}
\quad\mathcal{K}[\partial\Omega]=\frac{\tau^T\omega_{\tilde{x}\tilde{x}}\tau}{\sqrt{1+\omega_{\tilde{x}}^2}},\quad \tau=\tau[\partial\Omega],
\label{2}\end{equation}
and also $\mathcal K[\partial\Omega](M_0)=\omega_{\tilde{x}\tilde{x}}(\mathbf 0)$.

Let $W=W(x)$ be some $m$-Hessian kernel from Definition 6.6. Consider the following extension of $W\arrowvert_{\partial\Omega\cap B_r(M_0)}$ to $\bar\Omega_r$:
$$\widetilde{W}=W(\tilde x,\omega(\tilde x)).$$
Both of them attain maximum at $x=\mathbf0$. Therefore,
\begin{equation}\widetilde{W}_{\tilde x}=W_{\tilde x}+W_n\omega_{\tilde x},\quad-\widetilde{W}_{\tilde x\tilde x}(\mathbf0)=- W_{\tilde x\tilde x}(\mathbf0)-W_n(\mathbf0)\omega_{\tilde{x}\tilde{x}}(\mathbf 0)\in\bar K_n.\label{tlW}\end{equation}
On the other hand, Sylvester criterion (Lemma 5.3) guarantees that $W_{\tilde x\tilde x}\in K_{m-1}$ and due to monotonicity (\ref{mon}) relation (\ref{tlW}) guarantees the inclusion
$$-W_n(\mathbf0)\omega_{\tilde{x}\tilde{x}}(\mathbf 0)=-\widetilde{W}_{\tilde x\tilde x}(\mathbf0)+W_{\tilde x\tilde x}(\mathbf0)\in K_{m-1}.$$
In view of inequalities (\ref{pW}) the above line asserts, in particular, that $W_n(\mathbf0)<0$ and hence, $\mathcal K[\partial\Omega](M_0)=\omega_{\tilde{x}\tilde{x}}(\mathbf 0)\in K_{m-1}$, i.e., $\partial\Omega$ is $(m-1)$-convex at $M_0$.
\end{proof}
In order to construct an $m$-Hessian kernel we will keep to parametrization (\ref{1}) and consider a domain $\Omega^\beta_r\subset\Omega$ with $\partial\Omega^\beta_r=\Gamma_r^0\cup\Gamma_r^\beta$ (see Fig.3):
$$\Gamma_r^0=\{|\tilde x|\leqslant r, x^n=\omega(\tilde x)\}\subset\partial\Omega,$$
$$\Gamma_r^\beta=\left\{|\tilde x|\leqslant r, x^n=\omega^\beta(\tilde x)+\frac{\beta}{2}r^2\right\}\subset\bar\Omega,$$
where
\begin{equation}\label{hom}\omega^\beta=\omega^\beta(\tilde{x})=\omega(\tilde x)-\frac{\beta}{2}\tilde x^2,\quad \beta>0.\end{equation}
The choice of $r$ and the first restriction on parameter $\beta$ are presented in the following lemma.

\begin{lemma}
Let $\partial\Omega$ be $C^2$-smooth in some vicinity of $M_0\in\partial\Omega$. Assume that $\partial\Omega$ is $(m-1)$-convex at $M_0$ and $\mathbf k_{m-1}[\partial\Omega](M_0)\geqslant3\varepsilon>0$. Then there is $r=r(\varepsilon)$ such that $\Gamma_r^0$ is $(m-1)$-convex, $\mathbf k_{m-1}[\Gamma_r^0]\geqslant2\varepsilon$. Moreover, there exists $\beta_0=\beta_0(\varepsilon)$ such that for all $0<\beta\leqslant\beta_0$ the hypersurface $\Gamma_r^\beta$ is $(m-1)$-convex with respect to the normal directed into $\Omega\setminus\Omega_r^\beta$ and $\mathbf k_{m-1}[\Gamma_r^\beta]\geqslant\varepsilon$.
\end{lemma}

Indeed, the existence of $r(\varepsilon)$ is a straightforward consequence of $C^2$-continuity of $\partial\Omega$. As to $\beta_0(\varepsilon)$, we calculate $\mathcal K[\Gamma_r^\beta]$ by formulas (\ref{2}) with $\omega$ substituted on $\omega^\beta$ from (\ref{hom}) and arrive to $\mathbf k_{m-1}[\Gamma_r^\beta]=\mathbf k_{m-1}[\Gamma_r^0]+O(\beta)$. From now on the parameter $r$ is fixed.

\begin{remark}
The hypersurface $\partial\Omega^\beta_r=\Gamma_r^0\cup\Gamma_r^\beta$ from Lemma 6.8 is non-smooth. This is the reason we appoint normals by the rule from Remark 5.10 when speaking on $(m-1)$-convexity of $\partial\Omega_r$ smooth parts.
\end{remark}

Further on it is of use to imply another description of $\Omega_r^\beta$:
\begin{equation}\Omega_r^\beta=\left\{x\in\Omega:\:\:\frac{\beta}{2}\tilde x^2 <y(x)<\frac{\beta}{2}r^2,\:\: y=x^n-\omega^\beta(\tilde x)\right\}.\label{d0}\end{equation}
Notice that
\begin{equation}y_x=(-\omega^\beta_{\tilde x},1)\quad y_{ij}=-\omega^\beta_{ij}\quad i,j=1,\dots,n-1,\quad y_{ni}=y_{nn}=0.\label{yxx}\end{equation}
\begin{equation}\frac{1}{2}\beta r^2<\beta r^2-y(x)<\beta r^2, \quad x\in\Omega_r^\beta.\label{yb}\end{equation}
Consider in the domain (\ref{d0}) an auxiliary function
\begin{equation}W^\beta=\frac{y}{\beta r^2}\left(\frac{y}{2\beta r^2}-1\right),\quad0<\beta\leqslant\beta_0.\label{W}\end{equation}
and compute
\begin{equation}W^\beta_x=\frac{y-\beta r^2}{\beta^2 r^4}y_x,\quad W^\beta_{xx}=\frac{1}{\beta^2 r^4}\left((\beta r^2-y)(-y_{xx})+y_x\times y_x\right).\label{hW}\end{equation}
The identity (\ref{kos}) and the $p$-homogeneity of $T_p$ carry out the following presentation:
\begin{equation}T_p[W^\beta]=\frac{(\beta r^2-y)^{p-1}}{(\beta^2 r^4)^p}\left((\beta r^2-y)T_p[-y]+T_p^{ij}(-y_{xx})y_iy_j\right),\quad 1\leqslant p\leqslant n.\label{TW}\end{equation}
Everything is ready to make the final choice of parameter $\beta$.

\begin{figure}
	\center{\includegraphics[width=10cm]{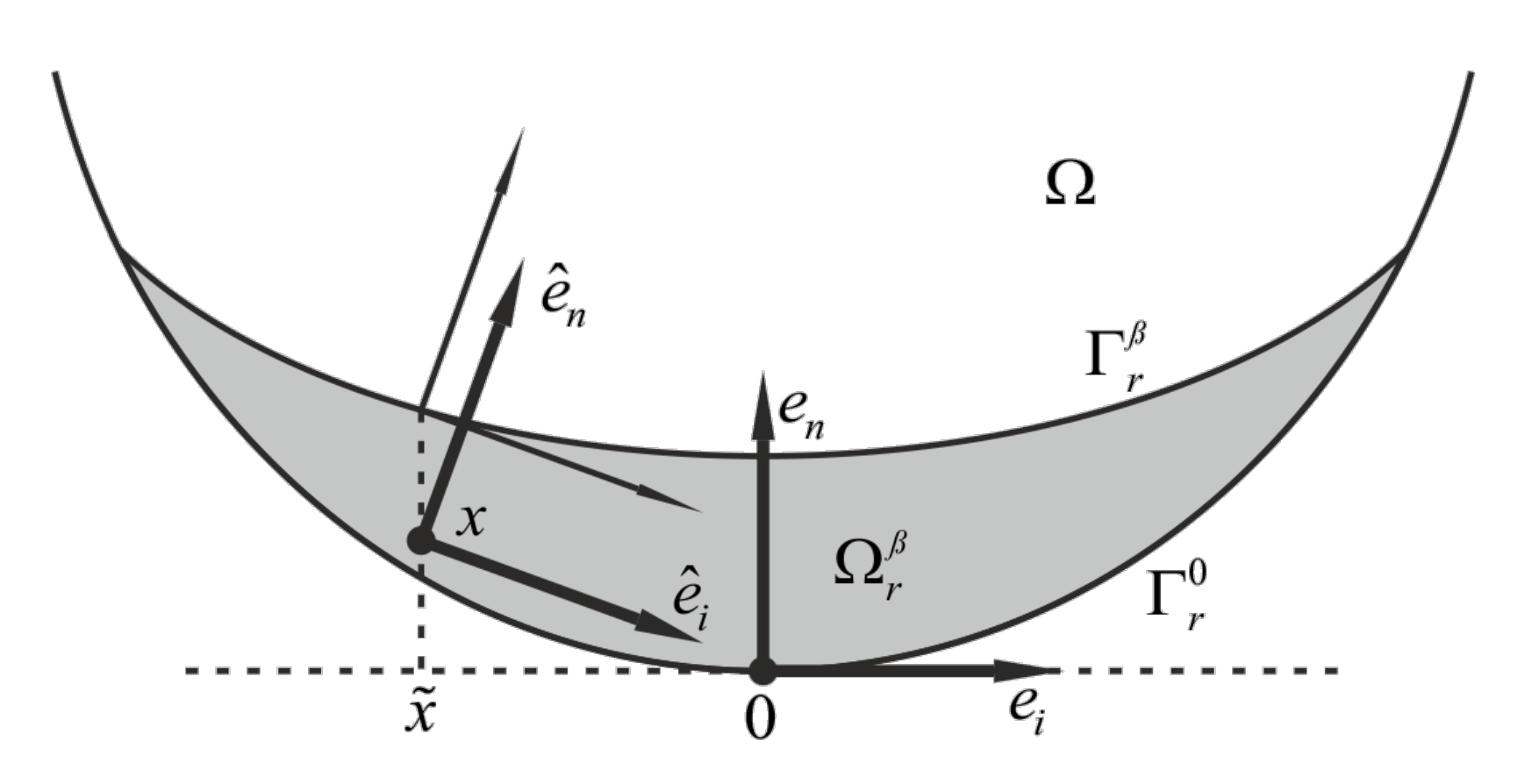} \\ Fig. 3}
\end{figure}

\begin{lemma}
Assume that the conditions of Lemma 6.8 are satisfied. Then there exists $\beta_1\leqslant\beta_0$ such that for all $0<\beta\leqslant\beta_1$ function (\ref{W}) is $m$-admissible in the domain $\bar\Omega_r^\beta$. Moreover,
\begin{equation}T_m[W^\beta](x)>\frac{\varepsilon}{2},\quad x\in\bar\Omega_r^\beta. \label{TmW}\end{equation}
\end{lemma}
\begin{proof} Our aim to find $\beta$ such that values (\ref{TW}) are positive in $\bar\Omega_r^\beta$ for $p=1,\dots,m$.
Due to orthogonal invariance (\ref{ort}) we are free to make use of the most convenient for computing cartesian basis in (\ref{TW}). In this course, fix up some point $M\in\bar\Omega_r^\beta$ with coordinates $x=(\tilde x,x^n)$ and relate to it the basis
$$\hat \mathbf e_i(M)=X_{(i)}[\Gamma^\beta_r](\tilde x),\quad i=1,\dots,n-1,\quad \hat\mathbf e_n(M)=\mathbf n^+[\Gamma^\beta_r](\tilde x).$$
This basis is the moving frame (\ref{j}) of hypersurface $\Gamma^\beta_r$ at $\tilde x$ (see Fig.3).
In view of parametrization (\ref{1}) with $\omega^\beta$ instead of $\omega$, this basis reads as
\begin{equation}\hat\mathbf e_i=(\tau^1_i,\dots,\tau^{n-1}_i,\omega^\beta_{(i)})(\tilde x),\quad\hat \mathbf e_n=\frac{(-\omega^\beta_{1},\dots,-\omega^\beta_{n-1},1)}{\sqrt{1+(\omega^\beta_{\tilde x})^2}}(\tilde x). \label{cb1}\end{equation}
Here $\tau=\tau[\Gamma^\beta_r](\tilde x)$ and we use the notations $\omega^\beta_{(i)}=\omega^\beta_k\tau^k_i$ similar to (\ref{j}).

Due to (\ref{d0}) the hypersurface $\Gamma^\beta_r$ is a level surface of the function $y$. So we have at the point $M$:
$$\frac{\partial y}{\partial\hat x^i}(M)=0, \quad i=1,\dots,n-1,\quad\frac{\partial y}{\partial \hat x^n}(M)=(y_x,\hat\mathbf e_n)=\sqrt{1+(\omega^\beta_{\tilde x})^2}.$$
Then relation (\ref{TW}) in the basis (\ref{cb1}) reads as
\begin{equation}T_p[W^\beta](M)=\frac{(\beta r^2-y)^{p-1}}{(\beta^2 r^4)^p}\left((\beta r^2-y)T_p[-y](M)+(1+(\omega^\beta_{\tilde x})^2)T^{nn}_p(-y_{\hat{x}\hat{x}})(M)\right)\label{Wcm}\end{equation}
It is easy to check that
\begin{equation}T^{nn}_p(-y_{\hat{x}\hat{x}})=T_{p-1}(-y^{\langle n\rangle}_{\hat{x}\hat{x}}),\label{Wm}\end{equation}
where $-y^{\langle n\rangle}_{\hat{x}\hat{x}}$ is the matrix derived from $-y_{\hat{x}\hat{x}}$ by crossing out the row and the column numbered by $n$. Calculate the matrix $-y^{\langle n\rangle}_{\hat{x}\hat{x}}$ via (\ref{yxx}), (\ref{cb1}):
$$-\frac{\partial^2y}{\partial \hat x^i\partial\hat x^j}=-\left(y_{xx}\hat\mathbf e_i,\hat\mathbf e_j\right)=
\omega^\beta_{kl}\tau^k_i\tau^l_j,\quad i,j=1,\dots,n-1.$$
By formula (\ref{2}) with $\omega^\beta$ instead of $\omega$ we get
\begin{equation}-y^{\langle n\rangle}_{\hat{x}\hat{x}}=\sqrt{1+(\omega^\beta_{\tilde x})^2}\:\mathcal K[\Gamma_r^\beta],\quad T_{p-1}(-y^{\langle n\rangle}_{\hat{x}\hat{x}})=
\left(1+(\omega^\beta_{\tilde x})^2\right)^{\frac{p+1}{2}}\mathbf k_{p-1}[\Gamma_r^\beta].\label{Wmm}\end{equation}

It follows from (\ref{Wcm}), (\ref{Wm}), (\ref{Wmm}) and (\ref{yb}) that
$$T_p[W^\beta](M)\geqslant\frac{(\beta r^2-y)^{p-1}}{(\beta^2 r^4)^p}\left(-\beta r^2|T_p[-y](M)|+\mathbf k_{p-1}[\Gamma_r^\beta](M)\right).$$
Due to (\ref{macl}) and Lemma 6.8 we have
$$\mathbf k_{p-1}[\Gamma_r^\beta]\geqslant\left(\mathbf k_{m-1}[\Gamma_r^\beta]\right)^{\frac{p-1}{m-1}}\geqslant\varepsilon,\quad 1\leqslant p\leqslant m.$$
Notice that $|T_p[-y](M)|$ depends on $||\Gamma_r^0||_{C^2}$ and is bounded with respect to $\beta$.
The latter guarantees the existence of $\beta_1=\beta_1(||\Gamma_r^0||_{C^2},\varepsilon)$ ensuring the inequalities
$$T_p[W^\beta](M)>\frac{\varepsilon}{2},\quad p=1,\dots,m,$$
for all $0<\beta\leqslant\beta_1$. But $M\in\bar\Omega_r^{\beta}$ has been an arbitrary point, which means that $W^{\beta}$ is an $m$-admissible function in $\bar\Omega_r^{\beta}$, $0<\beta\leqslant\beta_1$, and inequality (\ref{TmW}) holds true.
\end{proof}
Lemma 6.10 implies a supplement to Theorem 6.7.
\begin{Th}
Let $\partial\Omega$ be $C^2$-smooth in some vicinity of $M_0\in\partial\Omega$ and $(m-1)$-convex at $M_0$. Then an $m$-Hessian kernel of local sub-barriers at $M_0\in\partial\Omega$ does exist.
\end{Th}
\begin{proof}
Indeed, Lemma 6.10 confirms that function (\ref{W}) is $m$-admissible in the closure of domain (\ref{d0}) for all $0<\beta\leqslant\beta_1$. On the other hand, due to relations (\ref{yb}) a function $W=8/3W^\beta$ satisfies inequalities (\ref{cb}). Hence, the domain $\Omega^{\beta_1}_r$ and function
$$W=\frac{8}{3} W^{\beta_1}(x),\quad x\in\Omega^{\beta_1}_r,$$
with $r$, $\beta_1$ from Lemmas 6.8, 6.10, match Definition 6.6.
\end{proof}

\subsection{A sample of a priori estimate}
Now we demonstrate cooperation of Lemma 6.5 with Definition 6.6.

\begin{lemma}
Let $M_0\in\partial\Omega$, $\partial\Omega\cap B_{r_0}(M_0)$ be $C^2$-smooth and $(m-1)$-convex at $M_0$, $\mathbf k_{m-1}[\partial\Omega](M_0)\geqslant\varepsilon>0$.
Let $v\in C^2(\Omega\cap B_{r_0}(M_0))$ and denote the restriction of $v$ to $\partial\Omega$ by $\varphi$:
$$v(x)=\varphi(x),\quad x\in\partial\Omega\cap B_{r_0}(M_0).$$
Assume that there exist $\mu>0$, $u\in\mathbb K_m(\Omega\cap B_{r_0}(M_0))$ such that
\begin{equation}L[v;u]=F^{ij}_m[u]v_{ij}\leqslant\mu,\quad x\in\Omega\cap B_{r_0}(M_0).\label{mu}\end{equation}
Then
\begin{equation}-v_{\mathbf n}(M_0)\leqslant c\left(\frac{1}{\varepsilon},\:\mu,\:\|v\|_{C(\Omega\cap B_{r_0}(M_0))}, \:\|\partial\Omega\cap B_{r_0}(M_0)\|_{C^2},\:\|\varphi\|_{C^2(\partial\Omega\cap B_{r_0}(M_0))}\right),\label{vn}\end{equation}
where  $\mathbf n^+[\partial\Omega]$ is the interior to $\partial\Omega$ normal, $v_{\mathbf n}=(u_x,\mathbf n^+[\partial\Omega])$ is the normal derivative.
\end{lemma}
\begin{proof}
Let the parametrization of $\partial\Omega\cap B_{r_0}(M_0)$ be given by (\ref{1}).

Let $W$ be the $m$-Hessian kernel constructed in Subsection 6.2 and finally fixed in Theorem 6.11. Let $\Omega_r=\Omega^{\beta_1}_r\subset\Omega\cap B_{r_0}(M_0)$, $r\leqslant r_0$, be the corresponding sub-domain (\ref{d0}) for $W$, see Theorem 6.11.

Extend the function $\varphi$ from $\partial\Omega\cap B_{r_0}(M_0)$ to $\Omega_r$ by equality $\Phi(x)=\varphi(\tilde x,\omega(\tilde x))$, $x=(\tilde x, x^n)\in\Omega_r$.

Consider the following $\gamma$-set of functions:
\begin{equation}w_{\gamma}=\Phi(x)+\gamma W(x),\quad x\in\bar\Omega_r,\quad \gamma\geqslant1. \label{Aw}\end{equation}

Inequality (\ref{mu}) coincides with the first inequality (\ref{cps1}) from Lemma 6.5.

\begin{enumerate}
\item Firstly prove that Lemma 6.5 holds for $\Omega=\Omega_r$, $w=w_\gamma$ with sufficiently large $\gamma$. Write out the $p$-traces of function (\ref{Aw}) in the form
$$T_p[w_\gamma]=\gamma^p\:T_p\left(\frac{1}{\gamma}\Phi_{xx}+W_{xx}\right),\quad p=1,\dots,m.$$
It follows from Lemmas 6.8, 6.10 that there is
$$\gamma_1=\gamma_1\left(\frac{1}{\varepsilon},\:\mu,\:\|\partial\Omega\cap B_{r_0}(M_0)\|_{C^2},\|\varphi\|_{C^2(\partial\Omega\cap B_{r_0}(M_0))}\right)\gg1$$
such that $w_\gamma\in\mathbb K_m(\Omega_r)$ and $F_m[w]\geqslant\mu$ in $\Omega_r$ for all $\gamma\geqslant \gamma_1$. Then all requirements of Lemma 6.5 for functions $v$, $w=w_{\gamma}$ with $\gamma\geqslant\gamma_1$ are satisfied and the inequality (\ref{v-w}) holds.

\item Prove now that
\begin{equation}(w_\gamma-v)\arrowvert_{\partial\Omega_r}\leqslant 0.\label{bw}\end{equation}
By construction
$$(w_\gamma-v)\arrowvert_{\partial\Omega\cap\partial\Omega_r}=\gamma W\arrowvert_{\partial\Omega\cap\partial\Omega_r}\leqslant0.$$
Let $\gamma_2=\sup_{\Omega_r}(\Phi-v)$. Then for all $\gamma\geqslant \gamma_2$ the inequality
$$(w_\gamma-v)\arrowvert_{\Omega\cap\partial\Omega_r}=(\Phi-v)\arrowvert_{\Omega\cap\partial\Omega_r}-\gamma W\arrowvert_{\Omega\cap\partial\Omega_r}\leqslant \gamma_2-\gamma\leqslant0$$
holds true. Hence, relation (\ref{bw}) got valid with $\gamma\geqslant\gamma_2$.
\end{enumerate}

Let $\bar\gamma=\max\{\gamma_1,\gamma_2\}$. In presence of (\ref{v-w}) relation (\ref{bw}) brings out the estimate
$$w_{\bar\gamma}(x)\leqslant v(x),\quad x\in\bar\Omega_r.$$
In view of the equality $w_{\bar\gamma}(M_0)=\varphi(M_0)$ and (\ref{yxx}), (\ref{yb}), (\ref{hW}) there is the estimate
$$-v_{\mathbf n}(M_0)\leqslant \bar\gamma|W_n(M_0)|=\bar\gamma c\left(\frac{1}{\varepsilon}\right).$$
This guarantees the validity of (\ref{vn}).
\end{proof}

In order to apply Lemma 6.5 in further proceeding we reformulate the problem (\ref{hes}) for $u\in\mathbb K_m(\Omega)$ as
\begin{equation}F_m[u]=f>0,\quad u\arrowvert_{\partial\Omega}=\varphi.\label{Fm}\end{equation}
It is known that a priori estimate of $|u|$ in (\ref{Fm}) does not depend on geometric properties of $\partial\Omega$ but estimates of $|u_x|$ and $|u_{xx}|$ do depend. Here we demonstrate the estimation of $|u_x|$.

\begin{Th}
Let $\Omega$ be a bounded domain in $\mathbb R^n$. Assume that
\begin{equation}f\geqslant\nu>0,\quad x\in\Omega,\quad\mathbf k_{m-1}[\partial\Omega]\geqslant\varepsilon>0.\label{C1}\end{equation}
Then the inequality
\begin{equation}|u_x|\leqslant c\left(\frac{1}{\varepsilon},\:\frac{1}{\nu},\:\|u\|_{C(\Omega)},\:\|f\|_{C^1(\Omega)},\:\|\partial\Omega\|_{C^2},
\|\varphi\|_{C^2(\partial\Omega)}\right)\label{esC1}\end{equation}
holds valid in $\bar\Omega$ for all $m$-admissible solutions to the problem (\ref{Fm}).
\end{Th}
\begin{proof}
Denote
$$\mu=\sup_{\Omega}|f_x|.$$
Choose some vector $l\in\mathbb R^n$, $|l|=1$, and differentiate equation (\ref{Fm}) in the direction $l$:
\begin{equation}(\nabla_x F_m(u_{xx}),l)=F_m^{ij}[u]u_{lij}=(f_x,l)\leqslant\mu,\quad x\in\Omega.\label{dl}\end{equation}
On the other hand, the first inequality in (\ref{C1}) brings out
\begin{equation}F_m[w]\geqslant\mu,\quad w=\frac{\mu}{\nu}u,\quad x\in\Omega.\label{wu}\end{equation}
Relations (\ref{dl}), (\ref{wu}) coincide with inequalities (\ref{cps1}) of Lemma 6.5 with $v=u_l$, which via (\ref{v-w}) produces the estimate of $-u_l$ from above:
$$-u_l(x)\leqslant \frac{\mu}{\nu}\left(\sup_{\Omega}|u|+\sup_{\partial\Omega}|\varphi|\right)+\sup_{\partial\Omega}|u_l|,\quad x\in\Omega.$$
Since this reasoning works for an arbitrary vector $l$, this reduces the estimation of $|u_x|$ in $\bar\Omega$ to this estimate at $\partial\Omega$. In presence of given function $\varphi$, it suffices to find bounds for the normal derivative $u_{\mathbf n}=(u_x,\mathbf n^+[\partial\Omega])$ at $\partial\Omega$:
$$|u_x|\leqslant c\left(\frac{1}{\nu},\;\|u\|_{C(\Omega)},\;\sup_{\partial\Omega}|u_{\mathbf n}|,\;\sup_\Omega|f_x|,\;\|\varphi\|_{C^1(\partial\Omega)}\right).$$
Remind that $\mathbf n^+[\partial\Omega]$ is the interior to $\partial\Omega$ normal.
\begin{enumerate}
\item We start with estimation of $u_{\mathbf n}$ from above. Let $v$ be a harmonic in $\Omega$ function, $v\arrowvert_{\partial\Omega}=\varphi$. Since we are interested in $m$-admissible solutions $u$ to the problem (\ref{Fm}), the inequality (\ref{macl}) is in our possession:
    $$F_1[u]>F_m[u]=f\geqslant\nu>0.$$
    On the other hand
    $$\Delta v=F^{ij}_1[u]v_{ij}=0\leqslant v.$$
    The conditions of Lemma 6.5 are satisfied with $m=1$, $w=u$. Hence, $u\leqslant v$ in $\Omega$. Since $u\arrowvert_{\partial\Omega}=v\arrowvert_{\partial\Omega}$ we conclude that $u_{\mathbf n}\leqslant v_{\mathbf n}$, which gives an estimate from above for $u_{\mathbf n}$.

\item An a priori estimate of $u_{\mathbf n}$ from below is a consequence of Lemma 6.12. Indeed, inequality (\ref{mu}) is valid with  $v=u$, $\mu=\sup_{\Omega}f$:
$$L[u;u]=F^{ij}_m[u]u_{ij}=F_m[u]=f\leqslant\mu.$$
Therefore, the inequality (\ref{vn}) is valid for $v_{\mathbf n}=u_{\mathbf n}$ at an arbitrary point $M_0$ of $(m-1)$-convex surface $\partial\Omega$. This estimate of $u_{\mathbf n}$ from below concludes the proof of inequality (\ref{esC1}).
\end{enumerate}
\end{proof}

\end{document}